\documentclass[12pt]{amsart}
\usepackage{amsthm}
\usepackage{amsmath,amssymb,euscript,mathrsfs,bm}

\usepackage{pstricks}

\renewcommand{\part}[1]{\vspace{2em}\stepcounter{part} \begin{center}{\large\scshape Part~\Roman{part}:\hspace{5pt}  #1}\end{center}\vspace{0.5em}\addcontentsline{toc}{part}{Part~\Roman{part}: #1}}

\usepackage{eepic,bbm}

\usepackage{bbold}
\usepackage{dsfont}

\usepackage{ebgaramond}

\usepackage{setspace}
\usepackage[shortlabels]{enumitem}

\usepackage{tikz-cd}

\usepackage[colorlinks=true,linkcolor=black,citecolor=black,urlcolor=black]{hyperref}

\pushQED{\qed}

\psset{unit=1pt}
\psset{arrowsize=4pt 1}
\psset{linewidth=.5pt}

\countdef\x=23
\countdef\y=24
\countdef\z=25
\countdef\t=26

\def\graybox(#1,#2){
\x=#1 \y=#2 
\z=\x \t=\y
\advance\z by 10 
\advance\t by 10 
\psframe[fillstyle=solid,fillcolor=lightgray,linewidth=0pt](\x,\y)(\z,\t) 
\psline[linewidth=.5pt](\x,\y)(\x,\t)(\z,\t)(\z,\y)(\x,\y)}

\newcommand{\define}{\textbf}

\newcommand{\excise}[1]{}

\newcommand{\ul}{\underline}

\newcommand{\C}{\mathds{C}}

\newcommand{\Z}{\mathds{Z}}
\renewcommand{\P}{\mathds{P}}

\newcommand{\bV}{\mathds{V}}

\newcommand{\bX}{\mathds{X}}

\newcommand{\B}{\mathds{B}}
\newcommand{\E}{\mathds{E}}

\renewcommand{\O}{\mathcal{O}}

\newcommand{\pt}{\mathrm{pt}}

\DeclareMathOperator{\Sym}{Sym}
\DeclareMathOperator{\codim}{codim}

\DeclareMathOperator{\gr}{gr}

\DeclareMathOperator{\Hom}{Hom}

\DeclareMathOperator{\Span}{Span}

\DeclareMathOperator{\ev}{ev}

\newcommand{\Sgp}{\mathcal{S}}        

\newcommand{\tto}{\twoheadrightarrow}

\newcommand{\bull}{ {\footnotesize{\ensuremath{\bullet}}}  }

\newcommand{\uw}{\underline{w}}    
\newcommand{\tX}{\tilde{X}}        
\newcommand{\tY}{\tilde{Y}}
\newcommand{\tZ}{\tilde{Z}}

\newcommand{\cX}{\mathcal{X}}
\newcommand{\cY}{\mathcal{Y}}
\newcommand{\cZ}{\mathcal{Z}}

\newcommand{\tcX}{\tilde{\mathcal{X}}}
\newcommand{\tcY}{\tilde{\mathcal{Y}}}
\newcommand{\tcZ}{\tilde{\mathcal{Z}}}

\newcommand{\bbe}{\mathbb{e}}     
\newcommand{\bbc}{\mathbb{c}}

\newcommand{\bbOmega}{\mathbb{\Omega}}

\newcommand{\Gr}{ {Gr} }
\newcommand{\Fl}{ {Fl} }

\newcommand{\thGr}{\overline{\mathrm{Gr}}} 
\newcommand{\thFl}{\overline{\mathrm{Fl}}} 

\newcommand{\HH}{\mathcal{H}} 
\newcommand{\KK}{\mathcal{K}} 

\newcommand{\triple}{{\bm\tau}}

\newcommand{\enS}{\mathrm{S}}              
\newcommand{\Sch}{\mathfrak{S}}        
\newcommand{\bsS}{\overleftarrow{\mathfrak{S}}}  

\newcommand{\enG}{\mathrm{G}}              
\newcommand{\Groth}{\mathfrak{G}}        
\newcommand{\bsG}{\overleftarrow{\mathfrak{G}}}  

\newcommand{\ee}{\mathrm{e}}    

\newcommand{\tS}{S}             
\newcommand{\tbS}{\mathds{S}}   

\newtheoremstyle{scthm}%
{}{}{\itshape}{}{\scshape}{.}{ }{}

\newtheoremstyle{scdef}%
{}{}{}{}{\scshape}{.}{ }{}

\theoremstyle{scthm}

\newtheorem{theorem}{Theorem}[section]
\newtheorem{lemma}[theorem]{Lemma}
\newtheorem{proposition}[theorem]{Proposition}
\newtheorem{corollary}[theorem]{Corollary}

\newtheorem*{thm*}{Theorem}
\newtheorem*{cor*}{Corollary}
\newtheorem*{prop*}{Proposition}

\newtheorem*{thmA}{Theorem~A}
\newtheorem*{thmB}{Theorem~B}

\theoremstyle{scdef}
\newtheorem{definition}[theorem]{Definition}
\newtheorem{remark}[theorem]{Remark}
\newtheorem{example}[theorem]{Example}

\begin{document}

\newcommand{\isom}{\cong}
\renewcommand{\setminus}{\smallsetminus}
\renewcommand{\phi}{\varphi}
\newcommand{\exterior}{\textstyle\bigwedge}
\renewcommand{\tilde}{\widetilde}
\renewcommand{\hat}{\widehat}
\renewcommand{\bar}{\overline}

\title[Equivariant positivity and coproduct coefficients]{Strong equivariant positivity for homogeneous varieties and back-stable coproduct coefficients}
\author{David Anderson}\thanks{Partially supported by NSF CAREER DMS-1945212.}

\date{February 23, 2023}
\address{Department of Mathematics, The Ohio State University, Columbus, OH 43210}
\email{anderson.2804@math.osu.edu}

\maketitle

\renewcommand{\bfseries}{\itshape}

\begin{abstract}
Using a transversality argument, we demonstrate the positivity of certain coefficients in the equivariant cohomology and K-theory of a generalized flag manifold.  This strengthens earlier equivariant positivity theorems \cite{graham,agm} by further constraining the roots which can appear in these coefficients.

As an application, we deduce that structure constants for comultiplication in the equivariant K-theory of an infinite flag manifold exhibit an unusual positivity property, establishing conjectures of Lam-Lee-Shimozono \cite{LLS3}.  Along the way, we present alternative formulas for the back stable Grothendieck polynomials defined in {\it op. cit.}, as well as a new method for computing the coproduct coefficients.
\end{abstract}



\section*{Introduction}

Given a subvariety $Y$ of a homogeneous variety $G/P$, one can expand its cohomology class in the basis of Schubert classes.  A very basic application of Kleiman's transversality theorem \cite{kleiman} says that the coefficients in this expansion are all nonnegative integers.  When $Y$ is torus-invariant, one can ask for the expansion of its equivariant cohomology class.  Graham's theorem \cite{graham} says the coefficients are now polynomials in $\Z_{\geq 0}[\beta_1,\ldots,\beta_r]$, where the $\beta_i$ are simple positive roots.  Similar statements hold in K-theory \cite{brion} and equivariant K-theory \cite{agm}, where now positivity means the coefficients alternate in sign.

This article has two main aims.  The first is to establish stronger equivariant positivity theorems in cohomology and K-theory; the second is to apply these strong positivity theorems to prove conjectures of Lam, Lee, and Shimozono concerning back-stable Grothendieck polynomials.  Correspondingly, the two parts of this article are aimed at slightly different audiences, and may be read independently---the first part provides the general framework and theorems which are applied in the second part; the second part provides a motivating application, as well as examples illustrating the positivity theorems and the necessity of some of their hypotheses.

Positivity theorems---equivariant or not---may be understood as providing bounds on the coefficients in question.  When a precise formula for the coefficient is either unknown or impractical, such bounds can be useful.  From this point of view, it is natural to ask if the bounds can be improved and to seek sharper bounds if possible.

For example, Graham's theorem says that in the Schubert expansion of the equivariant class of a $T$-invariant subvariety $Y\subset G/P$, the coefficients belong to the cone $\Z_{\geq 0}[\beta_1,\ldots,\beta_r]$.  If one knows something more about $Y$, does it follow that these coefficients lie in a smaller cone?

Our answer begins with a simple idea: one way to prove Graham's theorem is to use matrices from the Borel group $B$ to move $Y$ into transverse position with respect to opposite Schubert varieties \cite{and07}.  So if $Y$ is invariant for some subgroup $A \subset B$, then a smaller, residual subgroup $B'\subset B$ should suffice to make $Y$ transverse, and one should see positivity in the roots of $B'$.

In fact, this naive idea is neither completely accurate nor sufficient for our purposes.\footnote{When $A$ is a subgroup of the unipotent radical $U\subset B$ which is normalized by $T$, a careful analysis of Graham's proof shows that the Schubert expansion of $[Y]$ is positive in the weights of $T$ acting on $U/A$.  (See, e.g., \cite[\S19.4]{ecag}.)  An application of this observation, along with an example of its limitations, is given in \cite[\S8]{infschub}.}  Correcting it leads to the key technical innovation of this article: the notion of {\it $S$-factorization}.  Given a subtorus $S\subset T$, an $S$-factorization for $B$ consists of a pair $(A,B')$, consisting of a unipotent subgroup $A$ of $G$ and a subgroup $B'$ of $B$, satisfying some conditions which are explained in \S\ref{s.factor}.  A basis for the character group of $S$ is {\it sufficiently positive} if all characters of $S$ acting on $B'$ are nonnegative sums of the basis elements; the role of this condition is discussed in \S\ref{s.gp}.

\begin{thmA}
Let $S\subset T$ be a subtorus, fix an $S$-factorization $(A,B')$ for $B$, and choose a sufficiently positive basis $\{\beta_1,\ldots,\beta_r\}$ for $S$ with respect to $B'$.  Let $Y\subset G/P$ be an irreducible subvariety which is both $S$-invariant and $A$-invariant.

\begin{enumerate}[(i)]
\item Writing the class
\[
  [Y] = \sum_w c_w \,[X_w]
\]
in the Schubert basis of $H_S^*(G/P)$, we have $c_w \in \Z_{\geq 0}[\beta_1,\ldots,\beta_r]$. \label{A-part-1}

\medskip

\item Suppose furthermore that $Y$ has rational singularities, and let $\partial Y\subset Y$ be a Cohen-Macaulay divisor which supports an ample line bundle, and is also invariant for both $S$ and $A$.  Writing
\[
 [\O_Y(-\partial Y)] = \sum_w d_w \,\xi_w
\]
in the ideal-sheaf basis of $K_S(G/P)$, we have
\[
(-1)^{\dim Y-\ell(w)} d_w \in \Z_{\geq 0}[\ee^{-\beta_1}-1,\ldots,\ee^{-\beta_r}-1].
\] \label{A-part-2}
\end{enumerate}
\end{thmA}

Parts \ref{A-part-1} and \ref{A-part-2} are proved as Theorem~\ref{t.mainH} and Corollary~\ref{c.mainK}, respectively, in \S\S\ref{s.transverse}--\ref{s.vanishing}.  The latter also includes a parallel statement in K-theory for the expansion of $[\O_Y]$ in the structure sheaf basis $[\O_{X_w}]$.

To get the sharpest results, one should take the largest possible unipotent group $A$ acting on $Y$.  In an $S$-factorization $(A,B')$, for a larger group $A$, one can take a smaller $B'$, making it easier for a basis to be sufficiently positive; then one sees more constraints on the coefficients.  At one extreme, if $A=U$ is the full unipotent radical of $B$, then one can take $B'=T$ to be the maximal torus, and in this case every basis for $S$ is sufficiently positive.  The assertion of Theorem~A(i) is trivially verified in this situation: the conditions imply $Y=X_v$ for some $v$, so $c_w = \delta_{v,w}$.  In K-theory, by contrast, the divisor $\partial Y$ need not be equal to $\partial X_v$, so it may happen that $[\O_Y(-\partial Y)]\neq \xi_v$.  At the other extreme, if $A$ is trivial, then one must take $B'=B$ and in this case the two parts of Theorem~A are the main results of \cite{graham} and \cite{agm}, respectively.

\medskip
The second part of this article is concerned with an application of the positivity theorems, which in fact motivated their discovery.  The setting involves certain infinite-dimensional Grassmannians and flag manifolds, $\thGr$ and $\thFl$.  There is a map $\boxplus\colon \thGr \times \thGr \to \thGr$, whose corresponding pullbacks induce coproduct structures on the equivariant cohomology and K-theory rings, $H_T^*\thGr$ and $K_T\thGr$.  A similar map $\boxplus\colon \thGr \times \thFl \to \thFl$ induces comodule strucures on $H_T^*\thFl$ and $K_T^*\thFl$.  The goal here is to establish a positivity property exhibited by the comultiplication structure constants with respect to Schubert bases.

To formulate the precise statement, let $T$ be a product of countably many copies of $\C^*$, with a basis of characters $y_i$ for $i\in\Z$.  Let $\ee^{y_i}$ be the corresponding multiplicative character, so the equivariant cohomology of a point is $H_T^*(\pt)=\Z[y_i : i\in \Z]$ and the representation ring is $R(T) = \Z[\ee^{\pm y_i} : i\in \Z]$.  Let $\Sgp_\Z$ be the group of permutations of $\Z$ which fix all but finitely many integers.  For each partition $\lambda$, there is a finite-codimensional Schubert variety $\Omega_\lambda \subset \thGr$; similarly, for each $w\in \Sgp_\Z$ there is a finite-codimensional $\Omega_w \subset \thFl$.  The classes of these Schubert varieties form bases for equivariant cohomology:
\[
  H_T^*\thGr = \bigoplus_\lambda H_T^*(\pt)\cdot [\Omega_\lambda] \quad \text{ and } \quad H_T^*\thFl = \bigoplus_w H_T^*(\pt)\cdot [\Omega_w].
\]
Likewise, structure sheaves of Schubert varieties constitute (formal) bases for equivariant K-theory:
\[
  K_T\thGr = \prod_\lambda R(T)\cdot [\O_{\Omega_\lambda}] \quad \text{ and } \quad K_T\thFl = \prod_w R(T)\cdot [\O_{\Omega_w}].
\]

In these bases, the comodule structures on $H_T^*\thFl$ and $K_T\thFl$ are governed by structure constants $\hat{c}_{\mu,v}^w$ and $\hat{d}_{\mu,v}^w$, respectively:
\begin{align*}
 \boxplus^*[\Omega_w] &= \sum_{\mu,v} \hat{c}_{\mu,v}^w \, [\Omega_\mu] \boxtimes [\Omega_v] & \text{in } H_T^*(\thGr \times\thFl); \\
 \boxplus^*[\O_{\Omega_w}] &= \sum_{\mu,v} \hat{d}_{\mu,v}^w \, [\O_{\Omega_\mu}] \boxtimes [\O_{\Omega_v}] & \text{in } K_T(\thGr \times\thFl).
\end{align*}
In the main body of the article, we often deal with {\it Schubert polynomials} $\enS_w$ and {\it Grothendieck polynomials} $\enG_w$, representing $[\Omega_w]$ and $[\O_{\Omega_w}]$, respectively, rather than the classes themselves.  These are polynomials (or series) in variable sets $c$, $x$, and $y$.  Formulas for them are given in \S\ref{s.polys}, and their relation with $[\Omega_w]$ and $[\O_{\Omega_w}]$ is explained via degeneracy loci in \S\ref{s.degloci}.  The above formulas translate to
\begin{align*}
 \Delta \enS_w(c;x;y) &= \sum_{\mu,v} \hat{c}_{\mu,v}^w \, \enS_{w_\mu}(c;x;y) \cdot \enS_v(c';x;y)
\intertext{and}
 \Delta \enG_w(c;x;z) &= \sum_{\mu,v} \hat{d}_{\mu,v}^w \, \enG_{w_\mu}(c;x;z) \cdot \enG_v(c';x;z) ,
\end{align*}
where $\Delta$ is the coproduct homomorphism, defined by sending $c_k$ to $\sum_{i=0}^k c_{k-i} \cdot c'_{i}$, and $z_i = 1-\ee^{-y_i}$.

A priori, $\hat{c}_{\mu,v}^w \in H_T^*(\pt)$ and $\hat{d}_{\mu,v}^w \in R(T)$.  The second main theorem of this paper is a positivity constraint.  

\begin{thmB}
Consider the total ordering $\prec$ on $\Z$ which puts all positive integers before all non-positive ones: $1\prec 2 \prec \cdots \prec -2 \prec -1 \prec 0$.  We have
\begin{align*}
 \hat{c}_{\mu,v}^w &\in \Z_{\geq 0}[y_j-y_i : i \prec j]
\intertext{and}
(-1)^{|\mu|+v-\ell(w)} \hat{d}_{\mu,v}^w &\in \Z_{\geq 0}[ \ee^{y_i-y_j}-1 : i \prec j ].
\end{align*}
\end{thmB}

When $v=e$ is the identity permutation, the coefficients $j_\mu^w := \hat{c}_{\mu,e}^w$ and $k_\mu^w := \hat{d}_{\mu,e}^w$ have special significance: they appear as coefficients in the {\it double Stanley functions} and {\it double K-Stanley functions} \cite[\S8]{LLS3}, respectively.  Positivity of $j_\mu^w$ was proved by Lam, Lee, and Shimozono in \cite[Theorem~4.22]{LLS}, by an argument which appeals to the quantum-affine correspondence and positivity in equivariant quantum cohomology.  The proof of Theorem~B given here includes a new---and somewhat more direct---proof of positivity of the cohomology coefficients $j^w_\mu$.  (See \S\ref{s.coprod}.)

In K-theory, Lam, Lee, and Shimozono conjectured that the double K-Stanley coefficients exhibit the (alternating) positivity of the second part of Theorem~B; that is, $(-1)^{|\mu|-\ell(w)}k_\mu^w$ lies in $\Z_{\geq 0}[ \ee^{y_i-y_j}-1 : i \prec j ]$ \cite[Conjecture~8.23]{LLS3}.  So Theorem~B includes a proof of their conjecture.\footnote{To compare with the notation of \cite{LLS3}, set $y_i=-\epsilon_i$, and follow the dictionary provided in \cite[\S2.8]{LLS3} to see $\ee^{y_i-y_j}-1$ maps to $-(a_i \ominus a_j)$.  See \S\ref{s.polys} for more detail.}

A K-theoretic quantum-affine correspondence has recently been established \cite{kato}, but no analogous positivity result is currently available in quantum K-theory.  In \cite[\S8]{infschub}, I gave a different argument which establishes a weaker form of positivity in cohomology, by applying Graham's theorem \cite{graham} to the image of the direct sum map, but these methods are insufficient to prove the full positivity theorem (even in cohomology) and do not apply in K-theory.  Hence the use of the strong transversality techniques in Theorem~A.

\medskip

The Schubert polynomials $\enS_w$ form a basis for the polynomial ring $\HH=\Z[c,x,y]$ as a $\Z[y]$-module.  The Grothendieck polynomials live in a certain formal series completion of this ring, which we denote $\hat{\HH}$, but they do not span this completed ring as an algebra over $R(T)$.  One is often interested in the $R(T)$-submodule $\KK \subset \hat{\HH}$ spanned by finite $R(T)$-linear combinations of $\enG_w$.  As shown in \cite[Proposition~8.27]{LLS3}, Theorem~B implies that any product $\enG_u\cdot\enG_v$ is a finite $R(T)$-linear combination of other $\enG_w$'s.  Equivalently:

\begin{cor*}[{\cite[Conjecture~8.27]{LLS3}}]
The submodule $\KK$ is an $R(T)$-subalgebra of $\hat{\HH}$.
\end{cor*}

\noindent
For similar reasons, the statements of \cite[\S8.7]{LLS3} hold unconditionally.  


\medskip
\noindent
{\it Acknowledgements.} The idea of looking for a stronger equivariant positivity theorem was indirectly inspired by a fact I learned from Allen Knutson: Graham's positivity theorem establishes that Schubert structure constants are sums of {\it squarefree} monomials in the positive roots.
\renewcommand{\bfseries}{\scshape}


\tableofcontents


\renewcommand{\bfseries}{\itshape}

\part{Transversality and strong positivity}

\section{Background}\label{s.bg}

Here we collect some of the basic facts we'll need, mainly to fix conventions and notation.  More details can be found in \cite{agm}, as well as \cite{ecag} (especially for equivariant cohomology) and \cite[\S5]{chriss-ginzburg} (for equivariant K-theory).

\subsection{Flag varieties}

Let $G \supset P \supset B \supset T$ be a semisimple linear algebraic group (over $\C$), with parabolic subgroup, Borel subgroup, and maximal torus.  Let $B^-$ be the opposite Borel group.  Let $R^+$ be the set of positive roots, i.e., the characters of $T$ acting on the Lie algebra of $B$.  Let $\Delta = \{\alpha_1,\ldots,\alpha_n\}$ be the set of simple roots.  
Unless otherwise specified, we generally assume the simple roots $\Delta$ are a basis for the character lattice of $T$.

Let $W=N_G(T)/T$ be the Weyl group, with parabolic subgroup $W_P = N_P(T)/T$.  The longest elements are $w_\circ \in W$ and $w_\circ^P \in W^P$.

We are mainly concerned with the generalized flag variety $X=G/P$, with the natural action of $T$ by left multiplication.  The $T$-fixed points $p_w$ in $X$ are indexed by minimal length coset representatives $w\in W$ for $W/W_P$.  One has Schubert varieties $X_w = \overline{B\cdot p_w}$ and opposite Schubert varieties $X^w=\overline{B^-\cdot p_w}$.  Our conventions are set up so that $\dim X_w = \codim X^w = \ell(w)$.  There is a boundary divisor $\partial = \partial X_w \subset X_w$, the union of all $X_v$ with $v\leq w$ and $\ell(v)=\ell(w)-1$.

Schubert varieties have rational singularities, and the boundary divisors $\partial X_w$ are Cohen-Macaulay.

\subsection{Bott-Samelson varieties}

We frequently need notation for the {\it balanced product}: let $H$ be a group acting on the right $Y$ and on the left on $Z$.  Then we have a quotient
\[
  Y \times^H Z := (Y \times Z)/(y\cdot h, z) \sim (y,h\cdot z).
\]

Now fix a reduced word $\uw = (i_1,\ldots,i_\ell)$ for $w$, so $w=s_{i_1}\cdots s_{i_\ell}$ is a minimal expression, and $\ell=\ell(w)$.  The Bott-Samelson variety is
\[
  \tX_{\uw} = P_{i_1}\times^B P_{i_2} \times^B \cdots \times^B P_{i_\ell}/B,
\]
where $P_i$ is the minimal parabolic containing $B$ whose Lie algebra also has the negative root $-\alpha_i$.  This is a smooth projective variety of dimension $\ell(w)$, and the multiplication map $\tX_{\uw} \to X=G/P$, sending $[p_1,\ldots,p_\ell]$ to $p_1\cdots p_\ell P$, is birational onto $X_w$.

There are similar Bott-Samelson desingularizations for the opposite Schubert varieties,
\[
 \phi\colon \tX^{\uw} \to X^w,
\]
but here $\uw$ should be a reduced word for $w_\circ w w_\circ^P$.  (The map to $X$ is the composition of $\tilde{X}_{\ul{w_\circ w w_\circ^P}} \to X_{w_\circ w w_\circ^P}$ with the involution $X \xrightarrow{w_\circ \cdot} X$, which sends $X_{w_\circ w w_\circ^P}$ to $w_\circ \cdot X_{w_\circ w w_\circ^P} = X^w$.)  We will mainly use these ``opposite'' Bott-Samelson varieties.

There are morphisms
\begin{align*}
  m\colon B \times X^w \to X, & & (b,x) \mapsto b\cdot x
  \intertext{and}
\tilde{m}\colon B\times\tX^{\uw} \to X, & & (b,x)\mapsto b\cdot\phi(x).
\end{align*}

We will need the following fact:
\begin{proposition}[{\cite[\S9]{agm}}] \label{p.smth1}
The morphism $m$ is flat and has normal fibers; the morphism $\tilde{m}$ is smooth.
\end{proposition}

The same is true, with the same proof, if $B$ is replaced by any parabolic $Q$ containing $B$.

\subsection{Mixing spaces and fiber bundles}\label{ss.mixing}

Let $S\subset T$ be a subtorus, and choose a basis $\beta_1,\ldots,\beta_r$ for its character lattice.  Using this basis, we form approximations to the classifying space of $S$.  For $m\gg0$, take $\E S = (\C^m\setminus 0)^{\times r}$, with $S$ acting (on the right) via the characters $\beta_1,\ldots,\beta_r$, so that
\[
  \mathds{B}S = \E S/S = (\P^{m-1})^{\times r} =: \P.
\]
We often need subvarieties
\[
  \P_J = \P^{j_1} \times \cdots \P_{j_r} \quad\text{ and }\quad\P^J = \P^{m-1-j_1} \times \P^{m-1-j_r},
\]
for a multi-index $J=(j_1,\ldots,j_r)$ of nonnegative integers; we have $\dim \P_J = \codim \P^J = |J| := j_1+\cdots+j_r$.  The subvariety $\P_J$ has a boundary divisor
\begin{equation}
\partial_J=\partial \P_J = \bigcup_{i=1}^r \P^{j_1}\times \cdots \times \P^{j_i-1} \times \cdots \times \P^{j_r}.
\end{equation}

The choice of basis also determines a {\it mixing space}
\[
 \cY = \E S\times^S Y := (\E S \times Y)/ (e\cdot s, y) \sim (e, s\cdot y),
\]
for any variety $Y$ with an action of $S$ (on the left).  Projection onto the first factor makes this a fiber bundle over $\P$, with fiber $Y$.  Generally, if a variety $Y$ is given, we use calligraphic font to denote the corresponding mixing space.  A subscript $J$ often denotes the restriction to $\P_J$, e.g., $\cY_J = \rho^{-1}\P_J \cap \cY$, where $\rho\colon \cX \to \P$ is projection.

An $S$-invariant effective divisor $\partial Y \subset Y$ determines a divisor $\partial\cY \subset \cY$, a fiber bundle over $\P$, as usual.  When the context is clear, we often write simply $\partial$ for such divisors.  Similarly abusing notation, we write $\partial_J \subset \cY_J$ for the pullback of $\partial\P_J$ under the projection $\cY_J \to \P_J$.

\begin{lemma}\label{l.tor0}
Let $\partial Y$ an $S$-invariant effective divisor, let $\partial=\partial\cY_J$ denote the corresponding divisor on $\cY_J$, and consider the divisor $\partial+\partial_J$.  Then
\[
  \O_{\cY_J}(-\partial-\partial_J) = \O_{\cY_J}(-\partial)\otimes \O_{\cY_J}(-\partial_J) \quad \text{and} \quad Tor^{\cY_J}_i( \O_{\cY_J}(-\partial), \O_{\cY_J}(-\partial_J) ) = 0
\]
for $i>0$.  If $Y$ and $\partial Y$ are Cohen-Macaulay, then
\[
  \omega_{\cY_J}(\partial+\partial_J) = \omega_{\cY_J}(\partial)\otimes \O_{\cY_J}(\partial_J) \quad \text{and} \quad  Tor^{\cY_J}_i( \omega_{\cY_J}(\partial), \O_{\cY_J}(\partial_J) ) = 0
\]
for $i>0$.
\end{lemma}

\noindent
(These statements follow directly from the fact that $\partial_J$ is a Cartier divisor.)

\subsection{Equivariant cohomology and K-theory}\label{ss.eqK}

We refer to \cite{ecag} and \cite{chriss-ginzburg} for details.

The $S$-equivariant cohomology of a point is the symmetric algebra of the character lattice: $H_S^*(\pt)=\Sym^*_\Z(M_S)$.  The equivariant K-theory of a point is the representation ring $R(S)$, which is naturally isomorphic to the group algebra $\Z[M_S]$.  Choosing a basis $\beta_1,\ldots,\beta_r$ for $M_S$, we have $H_S^*(\pt)=\Z[\beta_1,\ldots,\beta_r]$ and $R(S)=\Z[\ee^{\pm\beta_1},\ldots,\ee^{\pm\beta_r}]$.

The rings $H_S^*(X)$ and $K_S(X)$ are algebras over $H_T^*(\pt)$ and $R(S)$, respectively.  The Schubert classes $[X^w]$ form an additive basis for $H_S^*(X)$ as a module over $H_S^*(\pt)$ as $w$ runs over minimal coset representatives for $W/W_P$; the classes $[X_w]$ form another basis.  Likewise, the Schubert classes $\O^w = [\O_{X^w}]$ form a basis for $K_S(X)$ over $R(S)$, as do the {\it ideal sheaf classes} $\xi_w = [\O_{X_w}(-\partial)]$.  In K-theory, there are also bases formed by $\O_w = [\O_{X_w}]$ and $\xi^w = [\O_{X^w}(-\partial)]$.

Pushforward to a point determines a Poincar\'e pairing on equivariant cohomology, sending $(a,b)$ to $\int a\cdot b$.  In K-theory, product is derived tensor product, and the pushforward is given by the equivariant Euler characteristic, so the pairing is given by $\chi_S(X, a\cdot b)$.

With respect to the Poincar\'e pairing, the classes $[X^w]$ and $[X_w]$ form dual bases for $H_S^*X$; similarly, $\O^w$ and $\xi_w$ are dual bases for $K_S(X)$, as are $\O_w$ and $\xi^w$.  That is,
\[
 \int_X [X^w]\cdot [X_v]  = \delta_{w,v} \quad \text{ and }\quad  \chi_S( X,\, \O^w\cdot \xi_v) = \chi_S( X,\, \O_w\cdot \xi^v) = \delta_{w,v}
\]
in $H_S^*(\pt)$ and $R(S)$, respectively.

The finite-dimensional spaces $\cX$ may be used as approximations for computing equivariant cohomology and K-theory: that is, calculations in the $H_S^*(\pt)$-algebra $H_S^*(X)$ may be done in the $H^*(\P)$-algebra $H^*(\cX)$, and calculations in the $R(S)$-algebra $K_S(X)$ may be done in the $K(S)$-algebra $K(\cX)$.  In cohomology, this is standard and parallels constructions by Totaro and Edidin-Graham for equivariant Chow groups \cite{ecag}; for K-theory, it is explained in \cite[\S3]{agm}.

As $J$ ranges over multi-indices with $0\leq j_i\leq m-1$, the classes $[\P^J]$ and $[\P_J]$ form dual bases for $H^*(\P)$ (over $\Z$).  Likewise, the classes $\O^J = \O_{\P^J}$ and $\O_J(-\partial_J) = \O_{\P_J}(-\partial \P_J)$ form dual bases for $K(\P)$. 

Putting these observations together, suppose one writes the class of an invariant subvariety in $H_S^*(X)$ as
\[
  [Y] = \sum_{w,J} c_{w,J} \beta_1^{j_1}\cdots \beta_r^{j_r} \cdot [X_w],
\]
for some integers $c_{w,J}$.  Then
\begin{equation}\label{e.c1}
  c_{w,J} = \int_{\cX} [\cY]\cdot [\P_J]\cdot[\cX^w] ,
\end{equation}
where (as usual) $\cY \subset \cX$ is the subvariety corresponding to $Y \subset X$, and $\int_{\cX}\colon H^*\cX \to \Z$ is pushforward to a point.  The key idea is that, using $H^*(\P)$ to approximate $H_S(\pt)$, the class $[\P^J]$ corresponds to the monomial $\beta_1^{j_1}\cdots \beta_r^{j_r}$; using Poincar\'e duality, this monomial is picked out by cup product with the dual class $[\P_J]$.

Similarly, one may write the class of an equivariant sheaf in $K_S(X)$ as
\[
  [F] = \sum_{w,J} d_{w,J} (1-\ee^{-\beta_1})^{j_1}\cdots (1-\ee^{-\beta_r})^{j_r} \cdot \xi_w,
\]
or as
\[
  [F] = \sum_{w,J} p_{w,J} (1-\ee^{-\beta_1})^{j_1}\cdots (1-\ee^{-\beta_r})^{j_r} \cdot \O_{w},
\]
for integers $d_{w,J}$ and $p_{w,J}$.  Then
\begin{align}
  d_{w,J} &= \chi(\cX, \, [\mathcal{F}]\cdot [\O_J(-\partial_J)]\cdot[\O_{\cX^w}] ) & \text{and} \label{e.d1} \\
  p_{w,J} &= \chi(\cX, \, [\mathcal{F}]\cdot [\O_J(-\partial_J)]\cdot[\O_{\cX^w}(-\partial^w)] ),\label{e.p1}
\end{align}
where $\mathcal{F}$ is the sheaf on $\cX$ corresponding to the equivariant sheaf $F$.  As before, the idea is that the class $[\O_{\P^J}]$ corresponds to the ``monomial'' $(1-\ee^{-\beta_1})^{j_1}\cdots (1-\ee^{-\beta_r})^{j_r}$.

One of the main theorems of \cite{agm} says that for certain sheaves $F$, the integers $d_{w,J}$ and $p_{w,J}$ have predictable signs.  It requires the notion of a {\it positive basis} for the subtorus $S$: the basis $\{\beta_1,\ldots,\beta_r\}$ is {\it positive} if, for each positive root $\alpha\in R^+$, the restriction $\alpha|_S$ to $S$ is a nonnegative linear combination of the characters $\beta_i$.

\begin{theorem}[{\cite[Theorem~4.1]{agm}}]
Assume the subtorus $S$ has a positive basis, and fix such a choice.  Suppose $Y \subset X$ is an $S$-invariant subvariety with rational singularities, and $\partial = \partial Y$ is an $S$-invariant, Cohen-Macaulay effective divisor which supports an ample line bundle.  Expand $\O_Y(-\partial)$ as
\[
  [\O_Y(-\partial)]= \sum_{w,J} d_{w,J} (1-\ee^{-\beta_1})^{j_1}\cdots (1-\ee^{-\beta_r})^{j_r} \cdot \xi_w.
\]
Then
\[
  (-1)^{\dim Y - \ell(w) + |J|} d_{w,J} \in \Z_{\geq 0}.
\]
\end{theorem}

This can be rephrased equivalently as follows:

\begin{cor*}
For any positive basis $\beta_1,\ldots,\beta_r$ of $S$, and a variety $Y$ and divisor $\partial$ as above, in the expansion
\[
  [\O_Y(-\partial)] = \sum_w d_w\, \xi_w
\]
we have $(-1)^{\dim Y - \ell(w)} d_w \in \Z_{\geq0}[ \ee^{-\beta_1}-1,\ldots,\ee^{-\beta_r}-1]$.
\end{cor*}

Similar positivity statements hold for the integers $p_{w,J}$.  (See \cite[Theorem~4.1 and Corollary~5.1]{agm}.)

\section{$S$-factorizations}\label{s.factor}

Let $U'\subset B$ be a closed unipotent subgroup which is normalized by $T$, and let $B'=T\cdot U' \subset B$.  Let $A\subset G$ be a unipotent subgroup which is normalized by $S$, and let $A^-=A\cap B^-$.


\begin{definition}
The pair $(A,B')$ an {\it $S$-factorization} of $B$ if the multiplication map $A \times B' \times A^- \to G$ is dominant onto some parabolic $Q\supset B$, and for some $S$-invariant function $f\in \C[B'] \setminus \{0\}$, with nonvanishing locus $B'_\circ = \{b\in B'\,|\, f(b)\neq 0\}$, the map
\[
  A \times B'_\circ \times A^- \to Q
\]
is smooth.
\end{definition}

The $S$-invariance of the function whose nonvanishing defines $B'_\circ$ will play an important role in \S\ref{s.gp}, but not before.

\begin{example}\label{ex.s-fact}
In $G=SL_4$, with the diagonal torus $T$ and upper-triangular Borel $B$, consider the subgroups
\begin{align*}
  B' &= \left(\begin{array}{cccc}a & b & d & e \\0 & c & 0 & f \\0 & 0 & g & 0 \\0 & 0 & 0 & h\end{array}\right)
  \intertext{and}
  A &= \left(\begin{array}{cccc}1 & 0 & 0 & 0 \\ u & 1 & 0 & 0 \\0 & 0 & 1 & v \\0 & 0 & 0 & 1\end{array}\right).
\end{align*}
The map $A\times B' \times A^- \to G$ is dominant onto
\[
 Q = \left(\begin{array}{cccc}* & * & * & * \\ * & * & * & * \\0 & 0 & * & * \\0 & 0 & 0 & *\end{array}\right),
\]
and restricting to $B'_\circ = \{d\neq 0\}$, the map is smooth.  It is not smooth on the subset $d=0$.

For $S=T$, the pair $(A,B')$ is not an $S$-factorization, since the coordinate function $d$ has nontrivial $T$-character.

On the other hand, consider the subtorus $S \isom \C^* = \mathrm{diag}(s,s^{-1},s,s^{-1}) \subset T$, with character $y_1-y_2=y_3-y_4$.  (Here $y_i$ picks off the $i$th diagonal entry.)  Now the $S$-character of the coordinate $d$ is zero, and we have an $S$-factorization.  
\end{example}

Fix an $S$-factorization $(A,B')$, with $B'_\circ \subset B$ the corresponding dense open set.  We have a morphism
\[
m\colon  A \times B'_\circ \times X^w \to X,
\]
by $(a,b',x) \mapsto ab'\cdot x$, and similarly,
\[
 \tilde{m} \colon A \times B'_\circ \times \tX^{\uw} \to X,
\]
by $(a,b',x) \mapsto ab'\cdot\phi(x)$.

\begin{lemma}\label{l.smthfact}
The morphism $m$ is flat with normal fibers, and the morphism $\tilde{m}$ is smooth.
\end{lemma}

\begin{proof}
We will prove the second statement, the first being completely analogous.  Consider the diagram
\begin{equation}
\begin{tikzcd}
 A\times B'_\circ \times A^- \times \tX^w \ar[d,"f"] \ar[r,"p"] & A \times B' \times \tX^{\uw} \ar[d,"\tilde{m}"] \\
 Q \times \tX^{\uw} \ar[r,"g"] & X.
\end{tikzcd}
\end{equation}
Here $f$ is smooth, from the definition of $S$-factorization, and the map $g$ is smooth, by Proposition~\ref{p.smth1}.  So $g\circ f = \tilde{m}\circ p$ is smooth.  Since $p$ is smooth and surjective, being the projection away from $A^-$, it follows that $\tilde{m}$ is smooth (e.g., \cite[\href{https://stacks.math.columbia.edu/tag/02K5}{Lemma 02K5}]{stacks}).
\end{proof}

Next suppose we have a subvariety $Y \subset X$ which is invariant for $A$ and $S$.  Consider the fiber square
\begin{equation}
\begin{tikzcd}
 Z \ar[d,hook] \ar[r,"\mu"] \ar[dr,phantom,"\square"] & Y \ar[d,hook] \\
 A\times B'_\circ \times X^{w} \ar[r,"m"] & X.
\end{tikzcd}
\end{equation}
The inclusion of $Y$ and the action morphism $m$ are both $A$-equivariant, so $A$ also acts (equivariantly) on $Z$.  Because $A$ acts freely on $A\times B'_\circ \times X^{w}$, it follows that its action on $Z$ is also free, and that $Z \isom A\times Z'$, where $Z'=Z\cap (\{e\} \times B'_\circ \times X^w)$.  So the above diagram can be refined as follows:
\begin{equation}\label{e.square1}
\begin{tikzcd}
\{e\} \times Z'\ar[r,hook] \ar[d,hook] \ar[dr,phantom,"\square"] & A\times Z' \ar[d,hook] \ar[r,"\mu"] \ar[dr,phantom,"\square"] & Y \ar[d,hook] \\
\{e\} \times B'_\circ \times X^w \ar[r,hook]  & A\times B'_\circ \times X^{w} \ar[r,"m"] & X.
\end{tikzcd}
\end{equation}
By Lemma~\ref{l.smthfact}, $m$ is flat, and therefore $\mu$ is also flat.

Choosing an $A$- and $S$-equivariant desingularization $\psi\colon \tY \to Y$, we have a similar fiber diagram
\begin{equation}\label{e.square2}
\begin{tikzcd}
\{e\} \times \tZ'\ar[r,hook] \ar[d] \ar[dr,phantom,"\square"] & A\times \tZ' \ar[d] \ar[r,"\tilde\mu"]  \ar[dr,phantom,"\square"] & \tY \ar[d] \\
\{e\} \times B'_\circ \times \tX^{\uw} \ar[r,hook]  & A\times B'_\circ \times \tX^{\uw} \ar[r,"\tilde{m}"] & X.
\end{tikzcd}
\end{equation}
Since $\tilde{m}$ is smooth (by Lemma~\ref{l.smthfact} again), so is $\tilde\mu$.

\begin{lemma}\label{l.zprime}
With notation as in \eqref{e.square1} and \eqref{e.square2}, the induced map $f\colon \tZ' \to Z'$ is a desingularization, and we have
\[
  \dim Z' = \dim\tZ' = \dim B' + \dim Y - \ell(w).
\]
If $Y$ has rational singularities, so does $Z'$.
\end{lemma}

\begin{proof}
These statements closely parallel those of \cite[\S\S8--9]{agm}.  Since $\tilde\mu$ is smooth, $A\times \tZ'$ is nonsingular, and the map to $A \times Z'$ is proper and birational.  By \cite[Proposition~8.1]{agm}, $A\times Z'$ has rational singularities if $Y$ does.  The asserted properties of $\tZ'$ and $Z'$ follow.
\end{proof}

\section{The fiber group, sufficiently positive bases, and transversality}\label{s.gp}

The proofs of the results of \cite{agm} rely on the construction of a {\it fiber group}, which acts on the various fiber bundles over $\P$.  (This is a relatively simple instance of what is sometimes called a {\it gauge group} in mathematical physics.)  There is a group scheme $\mathcal{B} := \E S \times^S B \to \P$, where $S$ acts on $B$ by inverse conjugation.  Global sections form a connected algebraic group $\Gamma = \Hom(\P,\mathcal{B})$.  For each $x\in\P$, there is an evaluation homomorphism $\ev_x\colon \Gamma\to B$.  The group $\Gamma$ acts on $\cX$ through evaluation: $\gamma\cdot [x,gB] = [x,\ev_x(\gamma)\cdot gB]$.

The condition that the basis $\beta_1,\ldots,\beta_r$ be positive is equivalent to asking that $\ev_x$ be surjective for every $x\in\P$, that is, $\mathcal{B}$ is generated by global sections.  However, not every subtorus $S\subset T$ posseses a positive basis.  A major goal of this article is to develop a more flexible notion.

Let $B' = T\cdot U' \subset B$ be as in \S\ref{s.factor}.

\begin{definition}
A basis $\beta_1,\ldots,\beta_r$ for the character lattice of $S$ is {\it sufficiently positive} with respect to $B'$ if every character of $S$ acting on $U'$ is a nonnegative combination of the $\beta_i$.
\end{definition}

This is a weaker condition than positivity of the basis: we only ask that {\it some} of the positive roots restrict to nonnegative linear combinations of the $\beta_i$---namely, those roots which appear in $U'$.  In particular, any positive basis for $S$ is a sufficiently positive basis with respect to $B$, considered as a subgroup of itself.

The fiber group construction works for any $S$-normalized subgroup of $B$; in particular, applying it to $B'\subset B$, we have a subgroup scheme $\mathcal{B}'\subset \mathcal{B}$, and a subgroup of global sections $\Gamma'\subset\Gamma$.  Sufficient positivity guarantees that $\mathcal{B}'$ is generated by global sections, that is, the morphisms $\ev_x\colon \Gamma'\to B'$ are surjective for all $x\in\P$.

Given an open set $B'_\circ \subset B'$, let
\begin{equation}
\Gamma'_\circ = \bigcap_{x\in \P} \ev_x^{-1}(B'_\circ)
\end{equation}
be the subset of global sections of $\mathcal{B}'$ taking values in $B'_\circ$.  In general, $\Gamma'_\circ$ may be empty, but in the presence of $S$-factorization and sufficient positivity, the situation is good:

\begin{lemma}
Let $(A,B')$ be an $S$-factorization, with dense open $B'_\circ \subset B'$, and fix a sufficiently positive basis for $S$ with respect to $B'$.  Then $\Gamma'_\circ \subset \Gamma'$ is a dense open subset, and the evaluation map $\ev_x \colon \Gamma'_\circ \to B'_\circ$ is smooth and surjective for all $x\in\P$.
\end{lemma}

\begin{proof}
As a variety, $B' = C(S) \times U''$, where $Z(S) = \{ b\in B' \,|\, sbs^{-1}=b \}$ is the centralizer of $S$ in $B'$, and $U'' \subset U'$ is the closed subgroup where $S$ acts by nonzero characters.  Since $(A,B')$ is an $S$-factorization, $B'_\circ = \{f\neq 0\}$ for an $S$-invariant function $f$.  This means $B'_\circ = C(S)_\circ \times U''$, where $C(S)_\circ \subset C(S)$ is the nonvanishing locus of $f$ restricted to $C(S)$.  Since $S$ acts trivially on $C(S)$, we have
\[
  \mathcal{B}' = C(S) \times \mathcal{U}'',
\]
which evidently contains $\mathcal{B}'_\circ = C(S)_\circ \times \mathcal{U}''$ as a dense open subset.  Furthermore, letting $\Gamma'' = \Hom(\P,\mathcal{U}'')$, we have
\[
  \Gamma' = C(S) \times \Gamma'' \quad\text{and}\quad \Gamma'_\circ = C(S)_\circ \times \Gamma'',
\]
so $\Gamma'_\circ \subset \Gamma'$ is dense open, as claimed.

Since the chosen basis of characters of $S$ is sufficiently positive, each evaluation morphism $\ev_x \colon \Gamma' \to B'$ is surjective, and therefore smooth.  It follows that the restriction of $\ev_x$ to $\Gamma'_\circ$ is also smooth, and surjective onto $B'_\circ$.
\end{proof}

There is a morphism $\Gamma'\times\P \to \mathcal{B}'$ of group schemes over $\P$, defined by evaluating global sections.  By the above lemma, the corresponding map $\Gamma'_\circ \times\P \to \mathcal{B}'_\circ$ is smooth and surjective.

Fix a basepoint $x\in\P$, so we have the evaluation homomorphism $\ev_x \colon \Gamma'_\circ \to B'_\circ$.  Repurposing the letter $Z$, we have a fiber diagram
\begin{equation}\label{e.fiber-seq}
\begin{tikzcd}
Z \ar[d,hook] \ar[r,"q"] \ar[dr,phantom,"\square"]  & Z' \ar[d,hook] \ar[r,hook,"r"]  \ar[dr,phantom,"\square"] & A \times Z' \ar[d] \ar[r,"\mu"] \ar[dr,phantom,"\square"] & Y \ar[d,hook] \\
\Gamma'_\circ \times X^w \ar[r,"\ev_x"] & B'_\circ \times X^w \ar[r,hook] & A \times B'_\circ \times X^w \ar[r,"m"] & X.
\end{tikzcd}
\end{equation}
Let $\mu'\colon Z \to Y$ be the composition along the top sequence of arrows.

\begin{lemma}
Assume $Y$ is normal.  Let $\partial Y \subset Y$ be any $A$-invariant effective divisor, and consider its pullback $\partial Z = (\Gamma'_\circ \times X^w) \times_X \partial Y$.  Then $\O_Z(\partial Z) = (\mu')^*\O_Y(\partial Y)$.
\end{lemma}

\begin{proof}
For any $V \to Y$, we use the notation $\partial V$ to denote the pullback of $\partial Y$ to $V$.

By Lemma~\ref{l.smthfact}, $m$ is flat; hence so is $\mu$ and therefore $\O_{A\times Z'}(\partial(A\times Z')) = \mu^*\O_Y(\partial Y)$.

By $A$-invariance, we have $\partial(A\times Z') = A\times \partial Z'$.  Since $A$ is a unipotent group, it is isomorphic to affine space as a variety, and it follows that $\O_{Z'}(\partial Z') = r^*\O_{A\times Z'}(A \times \partial Z')$.

Finally, $\ev_x$ is smooth; hence so is $q$, and it follows that $\O_Z(\partial Z) = q^*\O_{Z'}(\partial Z')$.
\end{proof}

Applying the mixing space functor $\E S\times^S (\cdot)$ to the relevant part of Diagram~\eqref{e.fiber-seq}, and restricting to $\P_J$, we obtain another fiber diagram, in which we further recycle some notation for morphisms:
\begin{equation}\label{e.fiber2}
\begin{tikzcd}[back line/.style={densely dotted}]
& \cZ'_J  \ar{rd} \ar[hook,back line]{dd}  \\
\cZ_J \ar["q"]{ur}  \ar[crossing over,"\mu'"]{rr} \ar[hook]{dd} & & \cY_J \ar[hook, crossing over]{dd} \\
& \mathcal{B}_\circ'\times_\P \cX^w \ar[back line]{rd}  \\
\Gamma'_\circ \times \cX^w \ar["m'"]{rr} \ar[back line,"p"]{ur}   & & \cX 
\end{tikzcd}
\end{equation}
The map $p$ is induced by the map $\Gamma'_\circ \times\P \to \mathcal{B}'_\circ$, so both $p$ and $q$ are smooth and surjective.

\begin{lemma}\label{l.div-pullback}
Let $\partial Y \subset Y$ be an effective divisor which is invariant for $A$ and $S$, and let $\partial\cY_J \subset \cY_J$ and $\partial \cZ_J \subset \cZ_J$ be the induced divisors.  Then $\O_{\cZ_J}(\partial \cZ_J) = (\mu')^*\O_{\cY_J}(\partial \cY_J)$.
\end{lemma}

\begin{proof}
This follows from the previous lemma, since the statement is local, and $\cZ_J \to \cY_J$ is a map of fiber bundles over $\P_J$, whose fiberwise maps are isomorphic to $Z \to Y$.
\end{proof}

We conclude with some statements on dimensional transversality.  Choosing an equivariant desingularization $\tY \to Y$, we obtain a square
\begin{equation}\label{e.fiber3}
\begin{tikzcd}
\tcZ_J\ar[r,"\tilde\mu'"] \ar[d]  \ar[dr,phantom,"\square"]  & \tcY_J \ar[d] \\
 \Gamma'_\circ \times \tcX^w \ar[r,"\tilde{m}'"]   & \cX
\end{tikzcd}
\end{equation}
in the same way as the front face of Diagram~\eqref{e.fiber2}.

\begin{lemma}\label{l.czprime}
The induced map $f\colon \tcZ_J \to \cZ_J$ is a desingularization, and
\[
 \dim\tcZ_J = \dim \cZ_J = \dim \Gamma' + \dim Y + |J| -\ell(w).
\]
If $Y$ has rational singularities, so does $\cZ_J$.
\end{lemma}

The proof is the same as that of Lemma~\ref{l.zprime}, using $\dim \cY_J = \dim Y + |J|$.

\begin{corollary}\label{c.intersect}
For a sufficiently general $\gamma \in \Gamma'_\circ$, the intersection $\cY_J \cap \gamma\cdot \cX^w$ is reduced of dimension $\dim Y + |J| -\ell(w)$.  If $Y$ is Cohen-Macaulay, or has rational singularities, then $\cY_J \cap \gamma\cdot \cX^w$ has the same property.
\end{corollary}

\begin{proof}
Consider the diagram
\begin{equation}
\begin{tikzcd}
 & \cZ_J\ar[r,"\mu'"] \ar[d,hook] \ar[dl,"\pi",swap] \ar[dr,phantom,"\square"]  & \cY_J \ar[d,hook] \\
\Gamma'_\circ & \Gamma'_\circ \times \cX^w \ar[l] \ar[r,"m'"]   & \cX.
\end{tikzcd}
\end{equation}
The morphism $\pi$ is surjective, with fiber $\pi^{-1}(\gamma) = \cY_J \cap \gamma\cdot \cX^w$.  By Lemma~\ref{l.czprime} and the theorem on dimension of fibers, the assertions on dimension and reducedness follow.  Cohen-Macaulayness and rational singularities likewise follow, e.g. by \cite[Lemmas~1 and 3]{brion}.
\end{proof}

\section{Positivity theorems}\label{s.transverse}

Now we come to the main theorems.  First we fix some notation and hypotheses. which will be in force throughout this section (and for the rest of Part I).

\begin{itemize}[label=\bull,itemsep=5pt]

\item We fix an $S$-factorization $(A,B')$ for $B$, and a basis $\{\beta_1,\ldots,\beta_r\}$ for $S\subseteq T$ which is sufficiently positive.

\item $Y \subset X$ is a closed subvariety which is invariant for both $S$ and $A$.

\item $\partial=\partial Y \subset Y$ is a Cohen-Macaulay effective divisor which supports an ample line bundle, and is also invariant for $S$ and $A$.

\item Given a general element $\gamma \in \Gamma_\circ$, let $\O_{\cY_J \cap \gamma\cdot \cX^w}(-\partial)$ be the ideal sheaf of $\partial \cY_J$ restricted to $\cY_J \cap \gamma\cdot \cX^w$.

\item Similarly, $\O_{\cY_J \cap \gamma\cdot \cX^w}(-\partial^w_\gamma)$ is the ideal sheaf of $\gamma\cdot\partial \cX^w$ restricted to $\cY_J \cap \gamma\cdot \cX^w$.
\end{itemize}

Next, we define coefficients $c_{w,J}$ by writing
\begin{equation}
  [Y] = \sum_{w,J} c_{w,J} \beta_1^{j_1}\cdots\beta_r^{j_r} \cdot [X_w]
\end{equation}
in $H_S^*(X)$.  In K-theory, we define coefficients $d_{w,J}$ and $p_{w,J}$ by writing
\begin{equation}
  [\O_Y(-\partial)] = \sum_{w,J} d_{w,J} (1-\ee^{-\beta_1})^{j_1} \cdots (1-\ee^{-\beta_r})^{j_r} \cdot \xi_w
\end{equation}
and
\begin{equation}
  [\O_Y] = \sum_{w,J} p_{w,J} (1-\ee^{-\beta_1})^{j_1} \cdots (1-\ee^{-\beta_r})^{j_r} \cdot \O_w
\end{equation}
in $K_S(X)$.  So $c_{w,J}$, $d_{w,J}$, and $p_{w,J}$ are integers.

\begin{lemma}\label{l.push}
Let $\int_{\cX}\colon H^*(\cX) \to \Z$ and $\chi\colon K(\cX) \to \Z$ be pushforwards to a point.  For a general $\gamma\in\Gamma_\circ$, we have
\begin{align*}
  c_{w,J} &= \int_{\cX}[ \cY_J \cap \gamma\cdot \cX^w ] , \\
  d_{w,J} &= \chi\left(  \cY_J \cap \gamma\cdot \cX^w ,\,  \O_{\cY_J \cap \gamma\cdot \cX^w}(-\partial-\partial_J)\right), &\text{and} \\
  p_{w,J} &= \chi\left(  \cY_J \cap \gamma\cdot \cX^w ,\,  \O_{\cY_J \cap \gamma\cdot \cX^w}(-\partial_\gamma^w-\partial_J)\right).
\end{align*}
\end{lemma}

Here, as in \S\ref{ss.mixing}, we use the notation $\partial_J$ also to refer to the pullback of the divisor $\partial\P_J$ under the projection $\cY_J \cap \gamma\cdot \cX^w \to \P_J$.

\begin{proof}
The group $\Gamma'$ is connected, so for any $\gamma\in\Gamma'$, we have $[\cX^w] = [\gamma\cdot \cX^w]$, $[\O_{\cX^w}] = [\O_{\gamma\cdot \cX^w}]$, and $[\O_{\cX^w}(-\partial^w)] = [\O_{\gamma\cdot \cX^w}(-\partial^w_\gamma)]$.

The statement for $c_{w,J}$ follows from Poincar\'e duality (Eq.~\eqref{e.c1}), once we show $[\cY]\cdot[\P_J]\cdot[\cX^w] =  [ \cY_J \cap \gamma\cdot \cX^w ]$.  To see this equality, first recall that $\cY_J = \cY \cap \rho^{-1}\P_J$, where $\rho\colon \cX \to \P$ is the fiber bundle projection.  This intersection is evidently transverse, so $[\cY]\cdot[\P_J] = [ \cY_J ]$.  Now by Corollary~\ref{c.intersect}, we have $\dim(\cY_J \cap \gamma\cdot \cX^w) = \dim\cY_J - \ell(w)$, and therefore $[\cY_J]\cdot[\cX^w] =  [ \cY_J \cap \gamma\cdot \cX^w ]$, as required.

The statements in K-theory are proved in the same way, using Eqs.~\eqref{e.d1} and \eqref{e.p1}, and invoking Lemma~\ref{l.tor0} and \cite[Lemma~1]{brion} to deduce the equalities $[\O_{\cY}(-\partial)] \cdot[\O_J(-\partial_J)] \cdot [\O_{\cX^w}] = [\O_{\cY_J \cap \gamma\cdot \cX^w}(-\partial-\partial_J)]$ and $[\O_{\cY}] \cdot[\O_J(-\partial_J)]\cdot [\O_{\cX^w}(-\partial^w)] = [\O_{\cY_J \cap \gamma\cdot \cX^w}(-\partial^w_\gamma-\partial_J)]$ from the dimension formulas.
\end{proof}

The positivity theorem in cohomology is an immediate consequence:
\begin{theorem}\label{t.mainH}
For a general element $\gamma\in \Gamma'_\circ$, we have
\[
  c_{w,J} = \#( \cY_J \cap \gamma\cdot \cX^w )
\]
when the RHS is of expected dimension $0$ (and $c_{w,J} =0$ otherwise).  In particular, writing
\[
  [Y] = \sum_{w} c_{w} \cdot [X_w].
\]
in $H_S^*(X)$, we have $c_w \in \Z_{\geq0}[\beta_1,\ldots,\beta_r]$.
\end{theorem}

\begin{proof}
By Lemma~\ref{l.push}, we know $c_{w,J}$ is the degree of the effective $0$-cycle $[ \cY_J \cap \gamma\cdot \cX^w ]$, or is zero when the expected dimension of this cycle is not $0$.  (In fact this is already enough to conclude positivity.)  For sufficiently general $\gamma$, this intersection is reduced, so degree is equal to number of points.
\end{proof}

Positivity in K-theory comes in two flavors:

\begin{theorem}\label{t.mainK}
With notation and hypotheses as above, assume also that $Y$ has rational singularities.

\begin{enumerate}[(1)]
\item For a general element $\gamma\in \Gamma'_\circ$, we have
\[
  (-1)^{\dim Y-\ell(w)+|J|} d_{w,J} = \dim H^{\dim Y+|J|-\ell(w)}\big( \cY_J \cap \gamma\cdot \cX^w, \, \O_{\cY_J \cap \gamma\cdot \cX^w}(-\partial-\partial_J) \big).
\]
In particular, $(-1)^{\dim Y-\ell(w)+|J|} d_{w,J} \in \Z_{\geq 0}$. \label{mainK-1}

\item For a general element $\gamma\in \Gamma'_\circ$, we have
\[
  (-1)^{\dim Y-\ell(w)+|J|} p_{w,J} = \dim H^{\dim Y+|J|-\ell(w)}\big( \cY_J \cap \gamma\cdot \cX^w, \, \O_{\cY_J \cap \gamma\cdot \cX^w}(-\partial^w_\gamma-\partial_J) \big).
\]
In particular, $(-1)^{\dim Y-\ell(w)+|J|} p_{w,J} \in \Z_{\geq 0}$. \label{mainK-2}
\end{enumerate} 
\end{theorem}

The proof follows the same basic pattern as that of \cite[Theorem~4.2]{agm}, which in turn was based on \cite{brion}.  By Lemma~\ref{l.push}, the K-theoretic coefficient $d_{w,J}$ is equal to the sheaf Euler characteristic of $\O_{\cY_J \cap \gamma\cdot \cX^w}(-\partial)$.  To prove the theorem, we apply Kawamata-Viehweg vanishing to deduce the vanishing of sheaf cohomology of $\O_{\cY_J \cap \gamma\cdot \cX^w}(-\partial)$ in degrees other than $\dim Y-\ell(w)+|J|$.  This is carried out in the next section (Theorem~\ref{t.vanish}).

Before turning to the technicalities of the proof, we record an immediate consequence, which is our main application.

\begin{corollary}\label{c.mainK}
Assume the above hypotheses, including that $Y$ has rational singularities.  Write
\[
  [\O_Y(-\partial)] = \sum_{w} d_{w} \cdot \xi_w
\]
and
\[
  [\O_Y] = \sum_{w} p_{w} \cdot \O_w
\]
in $K_S(X)$.  Then
\[
  (-1)^{\dim Y - \ell(w)} d_w \quad \text{and} \quad (-1)^{\dim Y - \ell(w)} p_w
\]
lie in $\Z_{\geq 0}[\ee^{-\beta_1}-1,\ldots,\ee^{-\beta_r}-1]$.
\end{corollary}

\section{Vanishing theorems}\label{s.vanishing}

We continue to assume all the hypotheses of the previous section, including that $Y$ has rational singularities.  
Theorem~\ref{t.mainK} follows from a statement about vanishing of sheaf cohomology.  The goal now is to prove the following:

\begin{theorem}\label{t.vanish}
Fix a general element $\gamma\in\Gamma'_\circ$.
\begin{enumerate}[(1)]
\item For all $i<\dim(\cY_J \cap \gamma\cdot\cX^w) = \dim Y +|J| - \ell(w)$, we have
\[
  H^i\big( \cY_J \cap \gamma\cdot\cX^w, \, \O_{\cY_J \cap \gamma\cdot\cX^w}(-\partial-\partial_J) \big) = 0. 
\]
Equivalently, for all $i>0$, we have
\[
   H^i\big( \cY_J \cap \gamma\cdot\cX^w, \, \omega_{\cY_J \cap \gamma\cdot\cX^w}(\partial+\partial_J) \big) = 0. 
\]

\item For all $i<\dim(\cY_J \cap \gamma\cdot\cX^w) = \dim Y +|J| - \ell(w)$, we have
\[
  H^i\big( \cY_J \cap \gamma\cdot\cX^w, \, \O_{\cY_J \cap \gamma\cdot\cX^w}(-\partial_\gamma^w-\partial_J) \big) = 0. 
\]
Equivalently, for all $i>0$, we have
\[
   H^i\big( \cY_J \cap \gamma\cdot\cX^w, \, \omega_{\cY_J \cap \gamma\cdot\cX^w}(\partial_\gamma^w+\partial_J) \big) = 0. 
\]
\end{enumerate}
\end{theorem}

This is a refinement of \cite[Theorem~10.4]{agm}.  Because $\Gamma'$ is (in general) smaller than the full fiberwise group $\Gamma$ used in \cite{agm}, we do not have access to certain techniques, e.g., using flatness of the action morphisms.  We get around this by more efficiently exploiting the fiber bundle structure illustrated in Diagram \eqref{e.fiber2}.  (In some ways, this leads to a simpler argument.)

As in \cite{agm}, the equivalences in each part follow from Serre duality, since $\cY_J \cap \gamma\cdot \cX^w$ is Cohen-Macaulay, as are the relevant divisors.  We focus on proving the second statement in each pair.

Consider again the diagram
\begin{equation}\label{e.fiberZ}
\begin{tikzcd}
 & \cZ_J\ar[r,"\mu'"] \ar[d,hook] \ar[dl,"\pi",swap] \ar[dr,phantom,"\square"]  & \cY_J \ar[d,hook] \\
\Gamma'_\circ & \Gamma'_\circ \times \cX^w \ar[l] \ar[r,"m'"]   & \cX.
\end{tikzcd}
\end{equation}
In this notation, we define two divisors on $\cZ_J$.  The first is
\[
 \partial= \partial\cZ_J = \partial\cY_J \times_{\cX} (\Gamma'_\circ \times \cX^w),
\]
and the second is
\[
 \partial^w = \partial^w\cZ_J = \cY_J \times_{\cX} (\Gamma'_\circ \times \partial\cX^w).
\]
As usual, $\partial_J$ denotes the pullback of $\partial\P_J$ to $\cZ_J$.  For general $\gamma \in\Gamma'_\circ$, we have
\[
  \omega_{Z_J}(\partial+\partial_J)|_{\pi^{-1}(\gamma)} \isom \omega_{\cY_J\cap\gamma\cdot\cX^w}(\partial+\partial_J),
\]
since $\omega_{\Gamma'_\circ}$ is trivial.  Similarly,
\[
 \omega_{Z_J}(\partial^w+\partial_J)|_{\pi^{-1}(\gamma)} \isom \omega_{\cY_J\cap\gamma\cdot\cX^w}(\partial^w_\gamma+\partial_J).
\]
To prove parts (1) and (2) of Theorem~\ref{t.vanish}, it therefore suffices to prove
\begin{align}
 R^i\pi_*\omega_{\cZ_J}(\partial+\partial_J) &= 0 \label{e.vanish1}
 \intertext{and}
 R^i\pi_*\omega_{\cZ_J}(\partial^w+\partial_J) &= 0, \label{e.vanish2} 
\end{align}
respectively.  We will prove \eqref{e.vanish1}, leaving the similar proof of \eqref{e.vanish2} to the reader.

Fix an $S$- and $A$-equivariant desingularization $\tilde{Y} \to Y$, such that the pullback of $\partial Y$ is a simple normal crossings divisor; this induces a desingularization $\phi\colon \tcY_J \to \cY_J$, along with a simple normal crossings divisor $\partial\tcY_J$.  By Lemma~\ref{l.czprime}, we obtain a desingularization $f\colon \tcZ_J \to \cZ_J$ by fiber product with the Bott-Samelson resolution, as in the diagram
\begin{equation}\label{e.tildeZ}
\begin{tikzcd}
 &\tcZ_J\ar[r,"\tilde\mu'"] \ar[d,"\tilde\iota"] \ar[ld,"\tilde\pi",swap] \ar[dr,phantom,"\square"]  & \tcY_J \ar[d] \\
\Gamma'_\circ &\Gamma'_\circ \times \tcX^w \ar[r,"\tilde{m}'"] \ar[l]  & \cX.
\end{tikzcd}
\end{equation}
Let $\tilde\partial = \partial\tcZ_J$ be the divisor pulled back from $\partial\tcY_J$.  Since these are smooth varieties, $\O_{\tcZ_J}(\tilde\partial) = (\tilde\mu')^*\O_{\tcY_J}(\partial\tcY_J)$.

\begin{lemma}\label{l.higher1}
$R^i\tilde\pi_*\omega_{\tcZ_J}(\tilde\partial+\partial_J) = R^i f_*\omega_{\tcZ_J}(\tilde\partial+\partial_J) = 0$ for $i>0$.
\end{lemma}

The proof is the same as that of \cite[Lemma~10.9]{agm}, using the fact that $\partial Y$ supports an ample line bundle on $Y$, and $\partial\P_J$ supports an ample line bundle on $\P_J$, so $\partial\cY_J+\partial_J$ supports an ample line bundle on $\cY_J$.  (Our $\O_{\tcZ_J}(\tilde\partial+\partial_J)$ is denoted $\mathcal{M}$ in {\it loc. cit.}.)

\begin{lemma}\label{l.push1}
$f_*\big(\omega_{\tcZ_J}(\tilde\partial+\partial_J) \big) = \omega_{\cZ_J}(\partial+\partial_J)$.
\end{lemma}

\begin{proof}
Since $\partial_J$ is pulled back from $\P_J$, it suffices to show $f_*\big(\omega_{\tcZ_J}(\tilde\partial) \big) = \omega_{\cZ_J}(\partial)$.  This is proved exactly as in the discussion leading up to \cite[Lemma~10.9]{agm}, using Lemma~\ref{l.div-pullback} for the isomorphism $\O_{\cZ_J}(\partial) = (\mu')^*\O_{\cY_J}(\partial)$.
\end{proof}

Now \eqref{e.vanish1} follows from the combination of Lemmas~\ref{l.higher1} and \ref{l.push1}.  The proof of \eqref{e.vanish2} is similar, following the relevant parts of \cite{agm}.  So Theorem~\ref{t.vanish} is proved.

\pagebreak

\part{Positivity of coproduct coefficients}

The second goal of this article---and the one which motivated the search for a stronger equivariant positivity theorem---is to establish the positivity of coefficients appearing in coproduct formulas for {\it enriched} (or {\it back stable}) {\it Schubert} and {\it Grothendieck polynomials}.  These polynomials represent Schubert classes in the equivariant cohomology and K-theory (respectively) of certain infinite flag varieties.

As explained in \cite{LLS,LLS3} (and in \cite{part1,infschub}), these polynomials are also universal representatives for degeneracy loci of vector bundles over an arbitrary base scheme, and may therefore be characterized by this property. This perspective was emphasized in work on quiver loci by Buch and collaborators \cite{buch2,bkty,bkty2}.

\section{Preliminaries}

We generally follow the notation of \cite{infschub} and \cite{part1}, which in turn comes from \cite{LLS}.  For technical details about the K-theory of infinite flag varieties, we refer to \cite{LLS3}.

\subsection{Permutations and partitions}

For each finite interval $(-m,m]$, we have the symmetric group $\Sgp_{(-m,m]}$ of permutations of this interval.  The group $\Sgp_\Z$ is the group of permutations of $\Z$ which fix all but finitely many integers; in other words,
\[
  \Sgp_\Z = \bigcup_m \Sgp_{(-m,m]},
\]
with respect to the evident inclusions.  This group is generated by the adjacent transpositions $s_i$ (swapping $i$ and $i+1$), for all $i\in\Z$.

Permutations $w$ in $\Sgp_\Z$ are written in one-line notation, usually using the smallest window possible.  For instance, $w = [3,2]=s_2$ has $w(2)=3$ and $w(3)=2$, and $w(i)=i$ for all other $i$.  Similarly, $w=[-1,2,1,-2,0]$ has $w(-2)=-1$, $w(-1)=2$, $w(0)=1$, $w(1)=-2$, and $w(2)=0$, fixing all other integers.

Each permutation $w$ in $\Sgp_\Z$ determines a \define{dimension function} $k_w\colon \Z\times \Z \to \Z_{\geq0}$ by
\[
  k_w(p,q) = \#\{ i\leq p \,|\, w(i)>q \}.
\]
(Since $w$ permutes only finitely many integers, this number is always finite.)  These functions give a way of computing Bruhat order on $\Sgp_\Z$: one has $w\leq v$ iff $k_w(p,q) \leq k_v(p,q)$ for all $p,q$.

A permutation $w$ in $\Sgp_\Z$ is \define{$0$-Grassmannian} (or just \define{Grassmannian} for our purposes) if it has (at most) one unique descent at $0$; that is, $w(i)<w(i+1)$ for all $i\neq 0$.  Such permutations are in bijection with partitions $\lambda = (\lambda_1\geq \lambda_2 \geq \cdots \geq \lambda_s\geq 0)$, by
\[
  \lambda_k = w(1-k)-1+k  \quad \text{ for } k>0,
\]
and
\[
 w(k) = \lambda_{1-k}+k  \quad \text{ for } k\leq 0,
\]
filling in the unused integers in increasing order for the values of $w(k)$ for $k>0$.  
The Grassmannian permutation corresponding to $\lambda$ is written $w_\lambda$.  It has the property that
\[ k_{w_\lambda}(0,\lambda_i-i)=i \]
for each $i>0$.

The \define{longest permutation} in $\Sgp_{(-m,m]}$ is
\[
  w_\circ^{(m,m]} = [m,m-1,\ldots,-m+1],
\]
that is, it writes the integers in $(-m,m]$ in decreasing order.

\subsection{Vector spaces and tori}

Let $V$ be a countable-dimensional vector space with basis $e_i$, for $i\in\Z$.  This comes with an action of $T=T_\Z = \prod_{i\in\Z} \C^*$, scaling the $i$th coordinate by the character $y_i$.

For any interval $[m,n]$, we have the subspace $V_{[m,n]}$ spanned by $e_i$ for $m\leq i\leq n$.  In particular, we have spaces
\[
  V_{\leq q} = \Span\{ e_i \,|\, i\leq q\} \quad\text{and}\quad V_{>q} = \Span\{ e_i \,|\, i>q\}.
\]
These define a \define{standard flag} $V_{\leq\bull}$ and \define{opposite flag} $V_{>\bull}$ in $V$.

Often we will restrict to $V_{(-m,m]}$ for some $m\gg0$, and in this case, we also use the notation $V_{\leq\bull}$ and $V_{>\bull}$ for the standard and opposite flags in this finite-dimensional vector space.

The Borel group $B$ of upper-triangular matrices stabilizes the standard flag $V_{\leq \bull}$, while the opposite Borel $B^-$ of lower-triangular matrices stabilizes the opposite flag $V_{>\bull}$.

\subsection{Flag varieties and Schubert varieties}

Detailed discussions of infinite-dimensional flag varieties may be found in \cite{LLS3} and \cite{infschub}.  For our purposes, a less sophisticated setup suffices.

For $m\gg0$, we consider Grassmannians of half-dimensional spaces $\Gr = \Gr(m,V_{(-m,m]})$, as well as partial flag varieties $\Fl = \Fl(m-n,\ldots,m+n;V_{(-m,m]})$.  Flags $E_\bull$ are indexed so that $E_p$ has dimension $m+p$.

A permutation $w\in \Sgp_{(-m,m]} \subset \Sgp_\Z$ determines a flag $E^w_\bull$ by taking $E^w_p$ to be the span of $e_{w(i)}$ for $i\leq p$; we write the corresponding point as $x_w\in \Fl$.  Minimal representatives have $w(i)<w(i+1)$ for $i<m-n$ and $i>m+n$.  Each such permutation also determines a pair of opposite Schubert varieties in $\Fl$,
\[
  X_w = \overline{B\cdot x_w} \quad\text{and}\quad \Omega_w = \overline{B^-\cdot x_w},
\]
meeting transversally in the point $x_w$, with
\[
 \dim X_w = \codim \Omega_w = \ell(w).
\]
The variety $\Omega_w$ can be described alternatively as a degeneracy locus
\[
  \Omega_w = \big\{ E_\bull \,\big|\, \dim(E_p\cap V_{>q}) \geq k_w(p,q) \text{ for all }p,q\big\}.
\]
(In the notation of Part~I, we have $\Omega_w = X^w$.)

As reviewed in \S\ref{ss.eqK}, the classes of $X_w$ and $\Omega_w$ form Poincar\'e dual bases for cohomology and K-theory of $\Fl$, as $w$ ranges over minimal representatives.

In particular, $\Gr$ is the case of $\Fl$ where $n=0$.  In this case, Schubert varieties are indexed by Grassmannian permutations $w_\lambda$, or equivalently by the partitions $\lambda$ themselves.

\subsection{Rings of polynomials and formal series}

Let $\Lambda = \Z[c] = \Z[c_1,c_2,\ldots]$ be the polynomial ring in countably many variables.  This is equipped with a grading that places $c_i$ in degree $i$, as well as a (descending) filtration which collects elements of degrees $\geq i$.  Let
\[
  \HH = \Lambda[x,y] = \Z[c,x,y]
\]
be the polynomial ring in three (countable) sets of variables, where the $x$ and $y$ variables all have degree $1$, and are indexed (as above) by $i\in\Z$.  Let
\[
  R=\Z[y]
\]
be the polynomial ring in the $y$ variables.

Given any graded ring $A$, let $A[\![\beta]\!]_{\gr}$ be the subring of the usual formal series ring $A[\![\beta]\!]$ spanned by homogeneous elements, where $\beta$ is given degree $-1$ as before.  For instance, if $a_1,a_2,a_3,\ldots$ is a sequence of elements of $A$ with $a_i$ in degree $i$, then $a_1 + a_2 \beta + a_3 \beta^2 + \cdots$ lies in $A[\![\beta]\!]_{\gr}$ (and has degree $1$).

Let
\[
 \hat\Lambda^\beta = \Z[c^\beta][\![\beta]\!]_{\gr} \quad \text{and}\quad \hat{\HH}^\beta = \Z[c^\beta,x,z][\![\beta]\!]_{\gr}
\]
using the grading which gives $c^\beta_i$ degree $i$, and each $x$ and $z$ variable degree $1$.  Let $\hat{R}^\beta = \Z[z][\![\beta]\!]_{\gr}$ be the corresponding ring of series in just the $z$ variables.

For economy of notation, we often omit the superscript in $c^\beta$ when it is understood from context.

Evaluating $\beta=0$ and $z_i=y_i$ determines a homomorphism $\hat{\HH}^\beta \tto \HH$.  As remarked above.

It is useful to view the $z$ variables as series in the $y$ variables, by
\[
  z_i = \beta^{-1}(\ee^{\beta y}-1) = y + \frac{1}{2}\beta y^2 + \frac{1}{6} \beta^2 y^3 + \cdots.
\]
So $z_i\mapsto y_i$ under $\beta\mapsto 0$, and this is compatible with the above evaluation $\hat{\HH}^\beta \tto \HH$.

On the other hand, the evaluation $\beta=-1$ sends $z_i$ to $1-\ee^{-y_i}$, and defines homomorphisms $R^\beta \to R(T) = \Z[\ee^{\pm y_i}:i\in\Z]$.

Occasionally we use {\it formal group notation}, writing
\begin{align*}
 u \oplus v &= u+v + \beta uv
\intertext{and}
 u\ominus v &= \frac{u-v}{1+\beta v}.
\end{align*}
In particular, $\ominus v = \frac{-v}{1+\beta v}$.  (So $u\ominus v = u \oplus (\ominus v)$.)

At $\beta=0$, the formal group operations are just usual addition and subtraction.  At $\beta=-1$, with $z_i=1-\ee^{-y_i}$ as above, one has $z_i\oplus z_j = 1-\ee^{-y_i-y_j}$ and $z_i\ominus z_j = 1-\ee^{-y_i+y_j}$.

\begin{remark}\label{r.beta}
The $\beta$ parameter is related to {\it connective K-theory} (see \cite[Appendix~A]{kdeg}).  In our context, however, it can be seen simply a grading parameter and a useful device for collecting both cohomology and K-theory together, by specializing at $\beta=0$ and $\beta=-1$, respectively.  More specifically, the $\beta=-1$ specialization loses no information, because for any element of $\Z[\![c,x,z]\!]$, and for any specified degree, there is a unique homogeneous element of that degree in $\hat{\HH}^\beta$ which lifts the given element under the $\beta\mapsto-1$ specialization.
\end{remark}

\section{Schubert and Grothendieck polynomials}\label{s.polys}

Following \cite{LLS} and \cite{LLS3}, along with the variations in \cite{part1} and \cite{infschub}, we will define Schubert polynomials $\enS_w(c;x;y)$ in variables
\[
  c=(c_1,c_2,\ldots), \quad x=(\ldots,x_{-1},x_0,x_1,\ldots), \quad \text{ and }\; y = (\ldots,y_{-1},y_0,y_1,\ldots),
\]
as well as Grothendieck polynomials $\enG^\beta_w(c;x;z)$, using $z$-variables in place of the $y$-variables.  While $\enS_w$ is indeed a polynomial, $\enG^\beta_w$ is a formal series.  We will give formulas for them first, and then describe the rings where they live, along with specializations relating them.

\subsection{Formulas for polynomials}
First we write the formulas for the longest element $w_\circ^{(-m,m]} = [m,m-1,\ldots,0,-1,\ldots,-m+1]$, with $\lambda = (2m-1,2m-2,\ldots,2,1)$.  The Schubert polynomial is
\begin{equation}\label{e.enSdef}
  \enS_{w_\circ^{(-m,m]}} = \det\big( c(i)_{\lambda_i-i+j} \big)_{1\leq i,j\leq 2m-1},
\end{equation}
where
\[
  c(i) =\begin{cases} \displaystyle{c \cdot \prod_{a=i-m+1}^0 (1-x_a t) \prod_{b=1}^{m-i} (1+y_b t)} & \text{for } 1\leq i<m; \\ 
  \displaystyle{c} &\text{for } i=m; \text{ and }\\ 
  \displaystyle{c \cdot \prod_{a=1}^{i-m} \frac{1}{1-x_a t} \prod_{b=m-i+1}^0 \frac{1}{1+y_b t} }& \text{for } m<i\leq 2m-1; \end{cases}
\]
$c$ stands for the series $1+c_1 t + c_2 t^2 + \cdots$; and $c(i)_k$ is the coefficient of $t^k$ in the series expansion.  (This is just the standard notation for manipulating Chern series.)  For example, with $m=2$ we have
\[
  c(1) = c\cdot (1-x_0 t)(1+y_1 t), \quad c(2) = c, \quad \text{and} \; c(3) = c\cdot \frac{1}{1-x_1 t} \cdot\frac{1}{1+y_0 t},
\]
so
\[
 c(1)_3 = c_3 + (-x_0+y_1)c_2 - x_0 y_1 c_1
\]
and
\[
 c(3)_2 = c_2 + (x_1-y_0) c_1 + (x_1^2 - x_1 y_0 + y_0^2).
\]
And
\[
  \enS_{[2,1,0,-1]} = \det \left(\begin{array}{ccc} c(1)_3 & c(1)_4 & c(1)_5 \\ c(2)_1 & c(2)_2 & c(2)_3 \\ 0 & 1 & c(3)_1\end{array}\right).
\]

For Grothendieck polynomials, some formulas simplify if we use the notation
\[
  \tilde{x}_a = \ominus x_a := \frac{-x_a}{1+\beta x_a}.
\]
The Grothendieck polynomial for $w_\circ^{(-m,m]}$ is
\begin{equation}\label{e.enGdef}
  \enG^\beta_{w_\circ^{(-m,m]}} = \det\left( \sum_{d\geq 0} \beta^d \binom{2m-i+d-1}{d} c^\beta(i)_{\lambda_i-i+j+d} \right)_{1\leq i,j\leq 2m-1},
\end{equation}
where
\[
  c^\beta(i) =\begin{cases} \displaystyle{c^\beta \cdot \prod_{a=i-m+1}^0 (1+\tilde{x}_a t) \prod_{b=1}^{m-i} (1+z_b t)} & \text{for } 1\leq i<m; \\ 
  \displaystyle{c^\beta} &\text{for } i=m; \text{ and }\\ 
  \displaystyle{c^\beta \cdot \prod_{a=1}^{i-m} \frac{1}{1+\tilde{x}_a t} \prod_{b=m-i+1}^0 \frac{1}{1+z_b t} }& \text{for } m<i\leq 2m-1; \end{cases}
\]
and $c^\beta(i)_k$ is again defined by collecting the coefficient of $t^k$.

As with their finite counterparts, the Schubert and Grothendieck polynomials for other permutations are computed by descending induction, using difference operators.  For $i\in\Z_{\neq 0}$, let $\partial_i$ be the usual divided difference operator (on $x$ variables), acting by
\begin{equation}
  \partial_i(f) = \frac{f - s_i(f)}{x_i-x_{i+1}},
\end{equation}
where $s_i$ swaps $x_i$ and $x_{i+1}$ in the polynomial (or series) $f$; the $c$ and $y$ variables are treated as scalars.  For $i=0$, the operator is defined by the same formula, except it also acts on $c$ variables by
\begin{equation}
 \partial_0(c_k) = c_{k-1} + x_1 c_{k-2} + \cdots + x_1^{k-1}.
\end{equation}
The isobaric difference operators $\pi_i$ are defined similarly, by
\begin{equation}
\pi_i(f) = \frac{(1+\beta x_{i+1})f - (1+\beta x_i)s_i(f)}{x_i-x_{i+1}},
\end{equation}
and for $i=0$,
\begin{equation}
 \pi_0(c^\beta_k) = \left(\sum_{i=0}^{k-1} (-\tilde{x}_1)^i ( c^\beta_{k-1-i} - \beta c^\beta_{k-i} ) \right) - \beta(-\tilde{x}_1)^k,
\end{equation}
where $\tilde{x}_1 = \frac{-x_1}{1+\beta x_1}$ as before.  Here the transposition $s_i$ acts as usual on polynomials in $x$ (swapping $x_i$ and $x_{i+1}$).  For $i\neq 0$, $s_i$ fixes the $c$ variables, and
\[
  s_0(c^\beta_k) = c^\beta_k + (\tilde{x}_0-\tilde{x}_1)\sum_{i=0}^{k-1} (-\tilde{x}_1)^i c^\beta_{k-1-i}.
\]
These operators satisfy Leibniz-type rules,
\begin{align*}
 \partial_i( f\cdot g) &= \partial_i(f) \cdot g + s_i(f)\cdot \partial_i(g)\\
 \intertext{and}
 \pi_i(f\cdot g) &= \pi_i(f)\cdot g + s_i(f)\cdot\pi_i(g) + \beta s_i(f)\cdot g,
\end{align*}
and together, these formulas suffice to define operators $\partial_i$ and $\pi_i$.  They are related by
\[
  \pi_i(f) = \partial_i \big((1+\beta x_{i+1})f\big).
\]

The inductive formulas are:
\begin{align*}
 \enS_{ws_i} &= \partial_i\enS_w
 \intertext{and}
 \enG^\beta_{ws_i} &= \pi_i\enG^\beta_w
\end{align*}
when $ws_i<w$ in Bruhat order.

\begin{example}
We have $\enS_{[1,0]} =  \enS_{s_0} = c_1$, so $\partial_0\enS_{s_0} = 1$ as expected.  Similarly, $\enG^\beta_{[1,0]} = \enG^\beta_{s_0} = c^\beta_1 +\beta c^\beta_2 + \beta^2 c^\beta_3 + \cdots = \sum_{i=1}^\infty \beta^{i-1} c^\beta_i$, and a pleasant exercise shows $\pi_0\enG^\beta_{s_0}=1$.
\end{example}

\begin{remark}
The Schubert polynomial $\enS_w$ is recovered from $\enG^\beta_w$ by setting $\beta=0$ and $z=y$.  The back stable double Grothendieck polynomials $\bsG_w$ of \cite{LLS3} are recovered from $\enG^\beta_w$ by setting $\beta=-1$, along with an evaluation of the $c^\beta$ and $z$ variables to be described in \S\ref{ss.rings} below.
\end{remark}

\begin{remark}
The ``usual'' or ``finite'' Schubert polynomials are obtained from $\enS_w(c;x;y)$ by evaluating $c=1$ (so $c_k=0$ for all $k>0$) and substituting $-y_i$ for $y_i$:
\[
  \Sch_w(x;y) = \enS_w(1;x;-y).
\]
Similarly, the finite Grothendieck polynomials are obtained by setting $c^\beta=1$ and $\beta=-1$:
\[
 \Groth_w(x;z) = \enG_w(1;x;z),
\]
where we write $\enG_w$ for the evalutation of $\enG^\beta_w$ at $\beta=-1$.  Our conventions are set up so that
\[
  \Groth_{[n,n-1,\ldots,1]}(x;z) = \prod_{i+j\leq n} (x_i+z_j-x_i z_j),
\]
for $n>0$.\footnote{The fact that our determinantal formulas specialize this way involves a Vandermonde-like identity; this is the ``dominant case'' in \cite[\S1]{kdeg}.}
\end{remark}

\begin{remark}
The \define{vexillary} permutations include the Grassmannian permutations, as well as the longest permutation $w_\circ$.  They can be described in terms of a \define{triple} $\triple$, which consists of three sequences of integers $(k_\bull,p_\bull,q_\bull)$ with
\[
   0<k_1<\cdots<k_s, \quad p_1 \leq\cdots \leq p_s, \quad \text{and }\, q_1\geq \cdots \geq q_s,
\]
such that the sequence $q_i-p_i+k_i$ is nonincreasing.  Given a triple, one defines a partition $\lambda$ by setting $\lambda_{k_i} = q_i-p_i+k_i$, and filling in the other parts minimally.  (That is, $\lambda_k =\lambda_{k_i}$ whenever $k_{i-1}<k\leq k_i$.)  The corresponding vexillary permutation $w=w(\triple)$ is the minimal one such that $k_w(p_i,q_i)=k_i$ for each $i$.  (The conditions on $k_\bull$, $p_\bull$, and $q_\bull$ guarantee that $w(\triple)$ is well defined.  See \cite{part1} for details on how a vexillary permutation arises from a triple.)

If $w=w(\triple)$ is the vexillary permutation corresponding to a triple $\triple =(k_\bull,p_\bull,q_\bull)$, the Grothendieck polynomial $\enG^\beta_w$ is given by a determinantal formula similar to \eqref{e.enGdef} (see \cite{kdeg}).  Specifically, let
\[
  c^\beta(k_i) = c^\beta\cdot \frac{ \prod_{a\leq 0} (1+\tilde{x}_a t) \prod_{b\leq q_i} (1+z_b t) }{\prod_{a\leq p_i} (1+\tilde{x}_a t) \prod_{b\leq 0}(1+z_b t) },
\]
with $c^\beta(k)=c^\beta(k_i)$ for $k_{i-1}<k\leq k_i$.  Let $\lambda$ be the corresponding partition.  With these substitutions, \eqref{e.enGdef} gives a formula for $\enG^\beta_w$.  Specializing to $\beta=0$ recovers determinantal formulas for $\enS_w$.

In particular, when $\triple=(k_\bull,p_\bull,q_\bull)$ is given by $k_i=i$, $p_i=-m+i$, and $q_i=m-i$ for $1\leq i\leq 2m-1$, the corresponding permutation is $w_\circ^{(-m,m]}$, the partition is $\lambda = (2m-1,\ldots,2,1)$, and the above formula for $c^\beta$ breaks up into the cases described earlier for $w_\circ$.

Every Grassmannian permuation $w_\lambda$ is vexillary, with $p=(0,\ldots,0)$, so the determinantal formula involves no $x$ variables.  In this case, a version of the determinantal formula is presented in \cite[\S9]{LLS3}.
\end{remark}

\subsection{The rings of Schubert and Grothendieck polynomials}\label{ss.rings}

Recall $\HH=\Lambda[x,y]$ and $R=\Z[y]$, and $\hat{\HH}^\beta = \Z[c^\beta,x,z][\![\beta]\!]_{\gr}$ and $\hat{R}^\beta = \Z[z][\![\beta]\!]_{\gr}$.

From the definitions, we see that $\enS_w$ lies in $\HH$.  In fact, as $w$ ranges over $\Sgp_\Z$, the polynomials $\enS_w$ form a basis for $\HH$ as an $R$-module.  (See \cite{LLS,part1}.)

Similarly, $\enG^\beta_w$ lies in $\hat{\HH}^\beta$.  In contrast to the Schubert polynomials, Grothendieck polynomials do not form a basis for $\hat{\HH}^\beta$.  They span the \define{subalgebra of Grothendieck polynomials}\footnote{We have used $\KK$ for the ring denoted $B$ or $B(x;a)$ in \cite{LLS3}, to avoid conflict with our notation for Borel groups.},
\[
  \KK^{\beta} = \bigoplus_{w\in\Sgp_\Z} R^\beta \cdot \enG^\beta_w,
\]
as a module over
\[
  R^\beta = \Z[z_i, (1+\beta z_i)^{-1}: i\in \Z],
\]
a subring of $\hat{R}^\beta$.  The fact that $\KK^\beta$ is closed under multiplication is nontrivial---as explained in the Introduction, it was conjectured in \cite{LLS3} and follows from Theorem~B.

As remarked above, evaluating $\beta=0$ sends $\enG^\beta_w$ to $\enS_w$.

The evaluation $\beta=-1$ defines homomorphisms $R^\beta \to R(T) = \Z[\ee^{\pm y_i}:i\in\Z]$ and $\KK^{\beta} \to \KK = \bigoplus_w R(T)\cdot \enG_w$, where $\enG_w$ is the evaluation of $\enG^\beta_w$ at $\beta=-1$.

To compare with the notation of \cite{LLS3} and recover the back stable double Grothendieck polynomials $\bsG_w(x;a)$ discussed there, make the following substitutions.  Set
\[
  c^\beta = \prod_{i\leq 0} \frac{ 1 + z_i t}{1 + \tilde{x}_i t},
\]
with $\tilde{x}_i = \ominus x_i = \frac{-x_i}{1+\beta x_i}$.  Then set $\beta=-1$ and $z_i = \ominus a_i = \frac{-a_i}{1-a_i}$.  (In particular, our $\ee^{y_i}$ corresponds to the character $\ee^{-\varepsilon_i}$ in \cite{LLS3}.)

\begin{lemma}\label{l.en2bs}
Under these evaluations, $\enG^\beta_w$ maps to $\bsG(x;a)$.
\end{lemma}

The proof consists in observing that appropriate evaluations of both $\enG^\beta_w$ and $\bsG(x;a)$ represent degeneracy loci, which in turn determine the polynomials.  This is carried out in the next section.

\begin{remark}
As noted above, for a Grassmannian permutation $w_\lambda$, the polynomials $\enS_{w_\lambda}$ and $\enG^\beta_{w_\lambda}$ involve no $x$ variables.  So we may write them as
\[
  \enS_\lambda(c;y) \in \Lambda[y] \quad\text{and}\quad \enG^\beta_\lambda(c;z) \in \widehat{\Lambda[z]}^\beta,
\] 
respectively, where $\widehat{\Lambda[z]}^\beta = \Z[c^\beta,z][\![\beta]\!]_{\gr}$.
\end{remark}

\section{Degeneracy loci}\label{s.degloci}

The polynomials $\enS_w(c;x;y)$ and $\enG_w(c;x;z)$ are characterized by the property that they represent certain degeneracy loci.  For $\enS_w$, this is explained in \cite{part1} (see also \cite[\S12.4]{ecag}); for $\enG_w$ it is implicit in \cite{kdeg}.  The connection between degeneracy loci and the back stable versions of these polynomials is explained in \cite[\S10]{LLS3}.  Here we review the facts relevant to our setup.

Suppose $X$ is a nonsingular variety with an action of $T$, equipped with a $T$-equivariant vector bundle $V$ of rank $2m$ for $m\gg0$, and flags of equivariant subbundles
\[
  E_\bull: \cdots \subset E_{-1} \subset E_0 \subset E_1\subset \cdots \quad\text{and}\quad F_\bull: \cdots \subset F_{1} \subset F_0 \subset F_{-1}\subset \cdots,
\]
indexed so that $E_p$ has rank $m+p$ and $F_q$ has rank $m-q$.  Given a permutation $w\in\Sgp_\Z$ (with $m$ large enough so that $w(i)=i$ for $|i|\geq m$), we have a degeneracy locus
\begin{equation}
 \Omega_w = \big\{x\in X\,\big|\, \dim(E_p\cap F_q) \geq k_w(p,q) \text{ for all }p,q \big\}.
\end{equation}
This is a closed, $T$-invariant subvariety of $X$, of codimension at most $\ell(w)$.  When the codimension is maximal, Schubert and Grothendieck polynomials give formulas for its class in equivariant cohomology and K-theory, respectively.

\begin{proposition}\label{p.deg}
Assume $\Omega_w\subset X$ has codimension $\ell(w)$.  In $H_T^*X$, make the evaluations
\begin{align*}
  c_k &= c_k(V-E_0-F_0), \\
  x_i &= c_1( (E_i/E_{i-1})^* ), \text{and} \\
  y_i &= c_1( F_{i-1}/F_i ).
\end{align*}
Then
\begin{equation}\label{e.degS}
  [\Omega_w] = \enS_w(c;x;y).
\end{equation}
Evaluating the variables in the same way, this time as K-theoretic Chern classes  and with $z_i$ in place of $y_i$, we have
\begin{equation}\label{e.degG}
  [\O_{\Omega_w}] = \enG_w(c;x;z)
\end{equation}
in $K_T(X)$.

Furthermore, $\enS_w$ and $\enG_w$ are uniquely determined by these formulas (for all such $X$, $V$, $E_\bull$ and $F_\bull$).
\end{proposition}

\begin{proof}
The formulas go back to ones appearing in \cite{bkty,bkty2}.  A precise statement of \eqref{e.degS} is in \cite{part1}; the vexillary case of \eqref{e.degG} is \cite[Theorem~1]{kdeg}, and the general case follows by applying difference operators.

To see that the degeneracy locus formulas determine the polynomials $\enS_w$ and $\enG_w$, it suffices to consider the partial flag variety $X=\Fl(m-n,\ldots,m+n;V)$, for $m\gg n$.  Take $E_\bull$ to be the tautological flag, and $F_\bull = V_{>\bull}$ to be the trivial (opposite) flag, so $\Omega_w$ is a Schubert variety.  Both $H_T^*X$ and $K_T(X)$ are quotients of $\Lambda[x,y]$, with defining relations only in large degrees, so any finite coefficient of a representative for $\Omega_w$ can be detected.
\end{proof}

The specified evaluations send $c$ to $\prod_{i\leq 0}\frac{1+y_i}{1-x_i}$ in cohomology, and to $\prod_{i\leq 0}\frac{1+z_i}{1+\tilde{x}_i}$ in K-theory.  These are the same evaluations which identify $\enS_w$ with $\bsS_w$ and $\enG_w$ with $\bsG_w$, so Lemma~\ref{l.en2bs} follows.

Schubert varieties $\Omega_w \subset \Fl$ and $\Omega_\lambda \subset \Gr$ are basic examples of degeneracy loci.  As noted above, here one takes $E_\bull$ to be the tautological flag and $F_\bull = V_{>\bull}$ to be the trivial flag, so that $y_i$ is a torus weight and $z_i = 1-\ee^{-y_i}$ in $R(T)$.

We will need a variation, following \cite[\S8]{infschub}.  Let $\bV = V\oplus V$ with $T$ acting diagonally.  It will be convenient to double the index set and write the standard basis vectors for $\bV$ as $\bbe_i$ and $\bbe_{i'}$, so that
\[
   \bbe_i = (e_i,0) \quad \text{and}\quad \bbe_{i'} = (0,e_i)
\]
for each integer $i$.  Consider flags $\bV_{\leq\bull}$ and $\bV_{>\bull}$ such that $\bV_{\leq q} = V_{\leq 0}\oplus V_{\leq q}$ and $\bV_{>q} = V_{>0}\oplus V_{>q}$ for each $q$.\footnote{This can be arranged by ordering the basis vectors $\bbe_i$ and $\bbe_{i'}$, in such a way that all the nonpositive unprimed indices come first (in any order), followed by all the primed indices in their usual order, and then finally the positive unprimed indices (in any order).}

Given $m\gg n$, we have the partial flag variety $\bX = \Fl(2m-n,\ldots,2m+n;\bV)$, and for $w \in \Sgp_{(-n,n]} \subset \Sgp_\Z$, we have a Schubert variety
\begin{equation}
 \bbOmega_w = \big\{ \E_\bull \,\big|\, \dim( \E_p \cap \bV_{>q} ) \geq k_w(p,q) \text{ for all }p,q \big\}.
\end{equation}
The degeneracy locus formula gives the same results for $\bbOmega_w$:

\begin{corollary}\label{c.double-deg}
\begin{align}
  [\bbOmega_w] &= \enS_w(\bbc;x;y) & \text{in } H_T^*\bX
\intertext{and}
  [\O_{\bbOmega_w}] & = \enG_w(\bbc;x;z) & \text{in } K_T\bX
\end{align}
where $\bbc = c(\bV-\E_0-\bV_{>0})$ is the Chern class in equivariant cohomology or K-theory.
\end{corollary}

\section{Coproduct and direct sum morphism}\label{s.coprod}

There is a \define{coproduct homomorphism}
\[
  \Delta\colon \Z[c]=\Lambda \to \Lambda \otimes_\Z \Lambda = \Z[c,c']
\]
defined by $c_k \mapsto \sum_{i=0}^k c_{k-i}\cdot c'_{i}$.  This extends linearly to a coproduct of $\Z[y]$-algebras
\[
  \Delta\colon \Lambda[y] \to \Lambda[y]\otimes_{R}\Lambda[y]
\]
as well as to a comodule homomorphism
\[
  \Delta\colon \HH \to \Lambda[y]\otimes_{R} \HH,
\]
where $R=\Z[y]$ and $\HH=\Lambda[x,y]$ as before.

The same formula defines a comodule homomorphism
\[
 \Delta\colon \hat{\HH}^\beta \to \hat{\Lambda[z]}^\beta \otimes_{\hat{R}^\beta} \hat{\HH}^\beta,
\]
where $\hat{\Lambda[z]}^\beta = \Z[c^\beta,z][\![\beta]\!]_{\gr}$, $\hat{R}^\beta=\Z[z][\![\beta]\!]_{\gr}$, and $\hat{\HH}^\beta = \Z[c^\beta,x,z][\![\beta]\!]_{\gr}$ as before.

With our usual conventions on gradings, each of the above versions of $\Delta$ is a homomorphism of graded modules.  We also have a comodule homomorphism at $\beta=-1$, where there is no longer a grading (but the induced filtration is preserved).

The main theorem of Part~II concerns the coefficients $\hat{c}_{\mu,v}^w(y)$ and $\hat{d}_{\mu,v}^w(z;\beta)$ appearing in
\begin{align*}
 \Delta\enS_w(c;x;y) &= \sum_{\mu,v} \hat{c}_{\mu,v}^w(y)\, \enS_{\mu}(c;y)\,\enS_v(c';x;y)
 \intertext{and}
 \Delta\enG^\beta_w(c;x;z) &= \sum_{\mu,v} \hat{d}_{\mu,v}^w(z;\beta)\, \enG^\beta_{\mu}(c;z)\,\enG^\beta_v(c';x;z).
\end{align*}
The coefficients $\hat{c}_{\mu,v}^w$ and $\hat{d}_{\mu,v}^w$ are homogeneous of degree $\ell(w)-|\mu|-\ell(v)$.  Here is our positivity theorem:

\begin{theorem}\label{t.main2}
Consider the ordering of the integers so that the positive integers precede the non-positive ones:
\[
  1\prec 2 \prec \cdots \prec -2 \prec -1 \prec 0.
\]
We have
\begin{align*}
  \hat{c}_{\mu,v}^w(y) &\in \Z_{\geq0}[y_j-y_i : i\prec j]
\intertext{and}
  (-1)^{|\mu|+\ell(v)-\ell(w)} \hat{d}^w_{\mu,v}(z;\beta) &\in \Z_{\geq0}[-\beta, -z_j\ominus z_i : i\prec j].
\end{align*}
\end{theorem}

Since $\hat{d}^w_{\mu,v}(z;\beta)$ is homogeneous of degree $\ell(w)-|\mu|-\ell(v)$, the second statement is equivalent to saying
\[
 \hat{d}^w_{\mu,v}(z;\beta) \in \Z_{\geq0}[\beta, z_j\ominus z_i : i\prec j].
\]
This version clearly specializes to the first statement at $\beta=0$.  On the other hand, as noted in Remark~\ref{r.beta}, $\hat{d}^w_{\mu,v}(z;\beta)$ is the unique homogeneous lift of $\hat{d}^w_{\mu,v} := \hat{d}^w_{\mu,v}(z;-1)$, and it suffices to establish the corresponding statement
\begin{equation}\label{e.KTpos}
 (-1)^{|\mu|+\ell(v)-\ell(w)} \hat{d}^w_{\mu,v} \in \Z_{\geq0}[\ee^{y_i-y_j}-1 : i\prec j].
\end{equation}
(Recall that $-z_j\ominus z_i = \ee^{y_i-y_j}-1$ at $\beta=-1$.)  This is what we will prove.

The proof consists of an application of Theorem~A (from Part~I) to the situation of \cite[\S8]{infschub}, which we now describe.  Given any $w\in\Sgp_\Z$, choose $m\gg n\gg 0$, large enough so that $w\in\Sgp_{(-n,n]}$.  Write $V=V_{(-m,m]}$ for the vector space spanned by $e_i$, for $-m<i\leq m$, and let $\bV=V\oplus V$.  As above, we write $\bbe_i = (e_i,0)$ and $\bbe_{i'} = (0,e_i)$ for the standard basis of $\bV$.  Let
\begin{eqnarray*}
 & \Gr(V) = \Gr(m,V), \quad \Gr(\bV) = \Gr(2m,\bV), \\
  &\Fl(V) = \Fl(m-n,\ldots,m+n;V), \;\text{and} \; \Fl(\bV) = \Fl(2m-n,\ldots,2m+n;\bV)
\end{eqnarray*}
be the indicated Grassmannians and partial flag varieties.

There is a \define{direct sum morphism}
\begin{align}
 \boxplus\colon \Gr(V)\times \Gr(V) \to \Gr(\bV)
\end{align}
sending a pair of subspaces $(E\subset V, F\subset V)$ to the subspace $E\oplus F \subset \bV$.  This is equivariant for the diagonal $T$-action on $\Gr(V)\times\Gr(V)$ and the action on $\Gr(\bV)$ induced by scaling both $\bbe_i$ and $\bbe_{i'}$ by the character $y_i$, i.e., the diagonal action on $\bV$.

As an algebra over $\Z[y]$, the equivariant cohomology ring of $\Gr(V)$ is generated by the Chern classes $c_k=c_k(V-V_{>0}-\tS)$, where $\tS\subset V$ denotes the tautological subbundle.  The same is true for $\Gr(\bV)$, writing $\bbc_k = \bbc_k(\bV-\bV_{>0}-\tbS)$, with $\tbS\subset\bV$ the tautological bundle.  So $H_T\Gr(V)$ and $H_T\Gr(\bV)$ are quotients of $\Lambda[y]$.  By choosing $m\gg0$, the relations are in high enough degree to be irrelevant.

We have $\boxplus^*\tbS= \tS\oplus \tS'$ as vector bundles on $\Gr(V)\times\Gr(V)$, where $\tS$ is the tautological bundle on the first factor and $\tS'$ is the one for the second factor.  So, by the Whitney sum formula, we have
\[
  \boxplus^*\bbc_k = c_k + c_{k-1}\cdot c'_1 + \cdots + c_1\cdot c'_{k-1} + c'_k,
\]
where $c$ and $c'$ are the corresponding Chern classes, from the first and second factors.  That is, $\boxplus^*=\Delta$ is the coproduct homomorphism on $\Lambda$.

There is an analogous direct sum morphism
\begin{align}
 \boxplus\colon \Gr(V)\times \Fl(V) \to \Fl(\bV),
\end{align}
and the pullback on equivariant cohomology agrees with the comodule homomorphism
\[
  \Delta\colon \HH \to \Lambda[y]\otimes_{\Z[y]} \HH
\]
described above.  Here $x_i = c_1( (S_i/S_{i-1})^* )$, where $S_\bull$ is the tautological flag of bundles on $\Fl(V)$, and $y_i$ is the character scaling the coordinate $e_i$, so $y_i=c_1(V_{>i-1}/V_{>i})$.

Writing $\bbc$, $c$, and $c'$ for the same Chern classes in equivariant K-theory, the direct sum pullback
\[
  \boxplus^*\colon K_T\Fl(\bV) \to K_T\Gr(V)\otimes_{R(T)} K_T\Fl(V)
\]
also agrees with the comodule homomorphism $\Delta$ on $\hat{\HH}^\beta$ (after quotienting by relations in high degree and evaluating at $\beta=-1$).

Consider Schubert varieties $\Omega_\mu\subset \Gr(V)$, $\Omega_v \subset \Fl(V)$, and $\bbOmega_w\subset \Fl(\bV)$ as in \S\ref{s.degloci}.  
By Proposition~\ref{p.deg} and Corollary~\ref{c.double-deg}, the classes of these Schubert varieties are represented by Schubert and Grothendieck polynomials.  So we have the following lemma:

\begin{lemma}
The  coefficients $\hat{c}_{\mu,v}^w$ and $\hat{d}_{\mu,v}^w$ are the same as those appearing in formulas for $\boxplus^*$.  That is,
\begin{align}
 \boxplus^*[\bbOmega_w] &= \sum_{\mu,v} \hat{c}_{\mu,v}^w \, [\Omega_\mu]\times [\Omega_v] \label{e.hatc}
\intertext{and}
 \boxplus^*[\O_{\bbOmega_w}] &= \sum_{\mu,v} \hat{d}_{\mu,v}^w\, [\O_{\Omega_\mu}]\times [\O_{\Omega_v}]. \label{e.hatd}
\end{align}
\end{lemma}

At this point, we fix an ordering on the basis vectors for $V$ and $\bV$.  The order on the $e_i$ is the usual one, coming from the standard order on the integers $-m+1,\ldots,m$.  With respect to this order, the Borel $B$ preserving the flag $V_{\leq\bull}$ is the standard one of upper-triangular matrices.  The opposite Borel $B^-$ preserves $\Omega_\mu$ and $\Omega_v$; these are transverse to $B$-invariant Schubert varieties $X_\mu \subset \Gr(V)$ and $X_v\subset \Fl(V)$, respectively.

The basis vectors for $\bV$ are ordered by
\[
  0,-1,\ldots,-m+1,(-m+1)',\ldots,-1',0',1',\ldots,m',m,m-1,\ldots,2,1.
\]
With respect to this order, the Borel $\B$ preserving the flag $\bV_{\leq\bull}$ is the one shown in in the middle diagram of Figure~\ref{f.borels}.  The opposite Borel $\B^-$ preserves $\bbOmega_w$, which is transverse to a $\B$-invariant Schubert variety $\bX_w\subset \Fl(\bV)$.

Applying Poincar\'e duality, \eqref{e.hatc} becomes
\begin{align}
 [Y] &= \sum_{w} \hat{c}_{\mu,v}^w \, [\bX_w] \label{e.hatc2}
\end{align}
where
\[
  Y=\boxplus(X_\mu\times X_v) \subset \Fl(\bV),
\] so $[Y] = \boxplus_*[X_\mu\times X_v]$.  Recalling that in K-theory, the dual classes are $\xi_v = [\O_{X_v}(-\partial X_v)]$ and $\bm\xi_w = [\O_{\bX_w}(-\partial \bX_w)]$, \eqref{e.hatd} becomes
\begin{align}
 [\O_Y(-\partial Y)] &= \sum_{w} \hat{d}_{\mu,v}^w\, \bm\xi_w ,\label{e.hatd2}
\end{align}
where the divisor is
\[
 \partial Y = \boxplus(\partial X_\mu \times X_v) + \boxplus(X_\mu\times \partial X_v),
\]
so $[\O_Y(-\partial Y)] = \boxplus_*[\xi_\mu\times \xi_v]$.  The goal is to apply Theorem~A to \eqref{e.hatc2} and \eqref{e.hatd2}, for an appropriate choice of sufficiently positive basis.

Let $\mathds{T}\subset \B \subset GL(\bV)$ be the standard torus, so $\mathds{T}\isom T \times T$.  Let $S\isom T \subset T\times T \isom \mathds{T}$ be the diagonal.  For a basis of the character lattice of $S$, choose
\[
  \{ -\alpha_i : i\neq 0\} \cup \{ \alpha_0 \},
\]
where $\alpha_i = y_i-y_{i+1}$ is the usual simple root.  Evidently
\[
  \Z_{\geq0}[\alpha_0, -\alpha_i : i\neq 0] = \Z_{\geq 0}[y_j-y_i : i\prec j].
\]
Since $-z_{i+2}\ominus z_{i} = (-z_{i+2}\ominus z_{i+1}) + (-z_{i+1}\ominus z_{i}) + (-\beta)(-z_{i+2}\ominus z_{i+1})(-z_{i+1}\ominus z_{i})$, we also have
\[
  \Z_{\geq0}[\ee^{-\alpha_0}-1, \ee^{\alpha_i}-1 : i\neq 0] = \Z_{\geq 0}[\ee^{y_i-y_j}-1 : i\prec j].
\]
So the assertions of Theorem~\ref{t.main2} follow from Theorem~A, once we verify the hypotheses of the latter theorem:

\begin{itemize}[label=\bull,itemsep=2pt]
\item The subgroup $\B' \subset \B$ is chosen to be generated by the torus $\mathds{T}$, along with entries in the following positions: $(i,j)$ for $i>j$ or $j>0\geq i$; $(i,j')$ for $0\geq i\geq j$ or $j>0\geq i$; $(i',j)$ for $i>j>0$ or $j>0\geq i$; $(i',j')$ for $j>0\geq i$.  This is depicted in Figure~\ref{f.borels}.

\item By construction, the chosen basis for the character lattice of $S$ is sufficiently positive for $\B'$.

\item The subvariety $Y\subset \Fl(\bV)$ is invariant for the action of the subgroup $B\times B$ in $GL(V)\times GL(V) \subset GL(\bV)$.  In particular, $Y$ is invariant for the subtorus $S$, as well as for the unipotent group $A$ consisting of upper-triangular matrices with $1$'s on the diagonal, and possibly nonzero entries in positions $(i,j)$ and $(i',j')$ for $i<j\leq 0$, or for $0<i<j$.  This is depicted in Figure~\ref{f.borels2}.

\item $Y$ has rational singularities, since it is isomorphic to $X_\mu\times X_v$.

\item The divisor $\partial Y$ is also invariant for both $S$ and $A$.

\item $\partial Y$ is Cohen-Macaulay and supports an ample line bundle, being isomorphic to the boundary divisor of the Schubert variety $X_\mu\times X_v \subset \Gr(V)\times\Fl(V)$.
\end{itemize}
It remains to check that the pair $(A,\B')$ is an $S$-factorization of $\B$.  This is the content of the following lemma, which concludes the proof of Theorem~\ref{t.main2}.

\begin{figure}
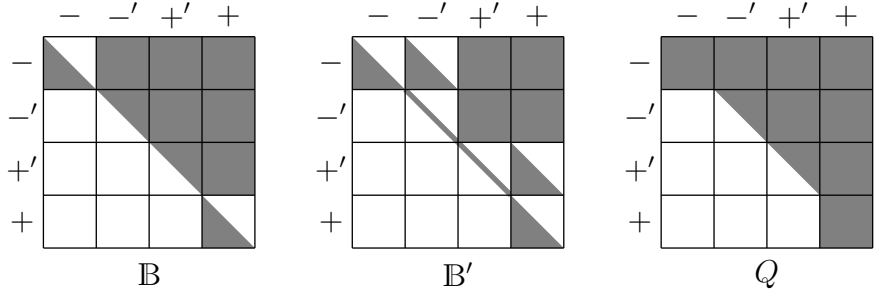

\begin{center}
\pspicture(0,0)(100,100)

\pspolygon*[linecolor=gray,fillcolor=gray](10,90)(10,70)(30,70)(10,90)
\pspolygon*[linecolor=gray,fillcolor=gray](30,90)(90,90)(90,70)(30,70)(30,90)

\pspolygon*[linecolor=gray,fillcolor=gray](30,70)(50,70)(50,50)(30,70)
\pspolygon*[linecolor=gray,fillcolor=gray](50,70)(90,70)(90,50)(50,50)(50,70)

\pspolygon*[linecolor=gray,fillcolor=gray](50,50)(70,50)(70,30)(50,50)
\pspolygon*[linecolor=gray,fillcolor=gray](70,50)(90,50)(90,30)(70,30)(70,50)

\pspolygon*[linecolor=gray,fillcolor=gray](70,30)(90,10)(70,10)(70,30)

\psline(10,10)(90,10)
\psline(10,30)(90,30)
\psline(10,50)(90,50)
\psline(10,70)(90,70)
\psline(10,90)(90,90)

\psline(10,10)(10,90)
\psline(30,10)(30,90)
\psline(50,10)(50,90)
\psline(70,10)(70,90)
\psline(90,10)(90,90)

\rput(2,20){$+$}
\rput(2,40){$+'$}
\rput(2,60){$-'$}
\rput(2,80){$-$}

\rput(20,98){$-$}
\rput(40,98){$-'$}
\rput(60,98){$+'$}
\rput(80,98){$+$}

\rput(50,0){$\B$}

\endpspicture
\hspace{.2in}
\pspicture(0,0)(100,100)

\pspolygon*[linecolor=gray,fillcolor=gray](10,90)(10,70)(30,70)(10,90)
\pspolygon*[linecolor=gray,fillcolor=gray](30,90)(30,70)(50,70)(30,90)
\pspolygon*[linecolor=gray,fillcolor=gray](50,90)(90,90)(90,70)(50,70)(50,90)

\psline[linecolor=gray,linewidth=2pt](30,70)(50,50)
\pspolygon*[linecolor=gray,fillcolor=gray](50,70)(90,70)(90,50)(50,50)(50,70)

\psline[linecolor=gray,linewidth=2pt](50,50)(70,30)
\pspolygon*[linecolor=gray,fillcolor=gray](70,50)(90,30)(70,30)(70,50)

\pspolygon*[linecolor=gray,fillcolor=gray](70,30)(90,10)(70,10)(70,30)

\psline(10,10)(90,10)
\psline(10,30)(90,30)
\psline(10,50)(90,50)
\psline(10,70)(90,70)
\psline(10,90)(90,90)

\psline(10,10)(10,90)
\psline(30,10)(30,90)
\psline(50,10)(50,90)
\psline(70,10)(70,90)
\psline(90,10)(90,90)

\rput(2,20){$+$}
\rput(2,40){$+'$}
\rput(2,60){$-'$}
\rput(2,80){$-$}

\rput(20,98){$-$}
\rput(40,98){$-'$}
\rput(60,98){$+'$}
\rput(80,98){$+$}

\rput(50,0){$\B'$}

\endpspicture
\hspace{.2in}
\pspicture(0,0)(100,100)

\pspolygon*[linecolor=gray,fillcolor=gray](10,90)(10,70)(30,70)(30,90)(10,90)
\pspolygon*[linecolor=gray,fillcolor=gray](30,90)(90,90)(90,70)(30,70)(30,90)

\pspolygon*[linecolor=gray,fillcolor=gray](30,70)(50,70)(50,50)(30,70)
\pspolygon*[linecolor=gray,fillcolor=gray](50,70)(90,70)(90,50)(50,50)(50,70)

\pspolygon*[linecolor=gray,fillcolor=gray](50,50)(70,50)(70,30)(50,50)
\pspolygon*[linecolor=gray,fillcolor=gray](70,50)(90,50)(90,30)(70,30)(70,50)

\pspolygon*[linecolor=gray,fillcolor=gray](70,30)(90,30)(90,10)(70,10)(70,30)

\psline(10,10)(90,10)
\psline(10,30)(90,30)
\psline(10,50)(90,50)
\psline(10,70)(90,70)
\psline(10,90)(90,90)

\psline(10,10)(10,90)
\psline(30,10)(30,90)
\psline(50,10)(50,90)
\psline(70,10)(70,90)
\psline(90,10)(90,90)

\rput(2,20){$+$}
\rput(2,40){$+'$}
\rput(2,60){$-'$}
\rput(2,80){$-$}

\rput(20,98){$-$}
\rput(40,98){$-'$}
\rput(60,98){$+'$}
\rput(80,98){$+$}

\rput(50,0){$Q$}

\endpspicture

\end{center}

\caption{Schematic depiction of groups $\B$, $\B'$, and $Q$.  The row and column labels indicate nonpositive ($-$), nonpositive primed ($-'$), positive primed ($+'$), and positive ($+$) indices; within each group, integers are ordered according to the standard order.  \label{f.borels} }
\end{figure}

\begin{lemma}\label{l.s-factor}
Let $\B'_\circ \subset \B'$ be the open subset having nonvanishing entries in positions $(i,i')$ for $i\leq 0$ and $(i',i)$ for $i>0$.  Let $A^- = A \cap \B^-$.  The multiplication morphism
\[
  A \times \B'_\circ \times A^{-} \to GL(\bV)
\]
is an open embedding into a parabolic subgroup $Q$ containing $\B$.
\end{lemma}

The parabolic group is depicted schematically in Figure~\ref{f.borels}.  The functions whose nonvanishing defines $\B'_\circ$ are $S$-invariant, and an open embedding is smooth and dominant, so the lemma indeed shows $(A,\B')$ is an $S$-factorization.

\begin{figure}
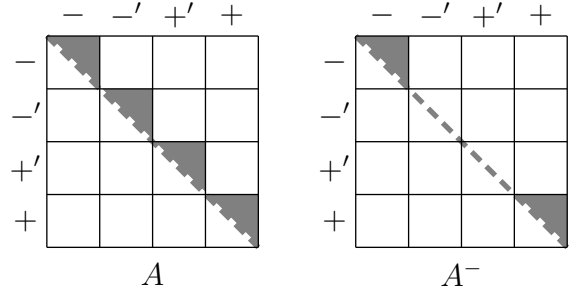

\begin{center}
\pspicture(0,0)(100,100)

\pspolygon*[linecolor=gray,fillcolor=gray](11,90)(30,90)(30,71)(11,90)

\pspolygon*[linecolor=gray,fillcolor=gray](31,70)(50,70)(50,51)(31,70)

\pspolygon*[linecolor=gray,fillcolor=gray](51,50)(70,50)(70,31)(51,50)

\pspolygon*[linecolor=gray,fillcolor=gray](71,30)(90,30)(90,11)(71,30)

\psline[linecolor=gray,linestyle=dashed,linewidth=2pt](10,90)(90,10)

\psline(10,10)(90,10)
\psline(10,30)(90,30)
\psline(10,50)(90,50)
\psline(10,70)(90,70)
\psline(10,90)(90,90)

\psline(10,10)(10,90)
\psline(30,10)(30,90)
\psline(50,10)(50,90)
\psline(70,10)(70,90)
\psline(90,10)(90,90)

\rput(2,20){$+$}
\rput(2,40){$+'$}
\rput(2,60){$-'$}
\rput(2,80){$-$}

\rput(20,98){$-$}
\rput(40,98){$-'$}
\rput(60,98){$+'$}
\rput(80,98){$+$}

\rput(50,0){$A$}

\endpspicture
\hspace{.2in}
\pspicture(0,0)(100,100)

\pspolygon*[linecolor=gray,fillcolor=gray](11,90)(30,90)(30,71)(11,90)

\pspolygon*[linecolor=gray,fillcolor=gray](71,30)(90,30)(90,11)(71,30)

\psline[linecolor=gray,linestyle=dashed,linewidth=2pt](10,90)(90,10)

\psline(10,10)(90,10)
\psline(10,30)(90,30)
\psline(10,50)(90,50)
\psline(10,70)(90,70)
\psline(10,90)(90,90)

\psline(10,10)(10,90)
\psline(30,10)(30,90)
\psline(50,10)(50,90)
\psline(70,10)(70,90)
\psline(90,10)(90,90)

\rput(2,20){$+$}
\rput(2,40){$+'$}
\rput(2,60){$-'$}
\rput(2,80){$-$}

\rput(20,98){$-$}
\rput(40,98){$-'$}
\rput(60,98){$+'$}
\rput(80,98){$+$}

\rput(50,0){$A^-$}

\endpspicture

\hspace{.2in}

\end{center}

\caption{Schematic depiction of groups $A$ and $A^-$. \label{f.borels2} }
\end{figure}

\begin{proof}
The groups $A$, $\B'$, and $A^-$ are all subgroups of $Q$, so the image of the multiplication map lies in $Q$.  The assertion about open embeddings comes from breaking the matrices up into $16$ $m\times m$ blocks and applying LDU factorization (i.e., top-cell Bruhat decomposition).

To spell this out, consider generic elements
\[
  a = \left(\begin{array}{c|c|c|c}a_{-,-} &  &  &  \\ \hline  & a_{-',-'} &  &  \\ \hline  &  & a_{+',+'} &  \\ \hline  &  &  & a_{+,+}\end{array}\right)
\quad \text{and} \quad \tilde{a} = \left(\begin{array}{c|c|c|c}\tilde{a}_{-,-} &  &  &  \\ \hline  & 1 &  &  \\ \hline  &  & 1 &  \\ \hline  &  &  & \tilde{a}_{+,+}\end{array}\right)\]
in $A$ and $A^-$, respectively.  Similarly, write
\[
  b = \left(\begin{array}{c|c|c|c}b_{-,-} & b_{-,-'} & b_{-,+'} & b_{-,+} \\\hline  & b_{-',-'} & b_{-',+'} & b_{-',+} \\\hline  &  & b_{+',+'} & b_{+',+} \\\hline  &  &  & b_{+,+}\end{array}\right)
\]
in $\B'_\circ$.  Each entry represents an $m\times m$ block, with rows and columns on the indices indexed by the subscripts.  The blocks $a_{\bull,\bull}$ and $\tilde{a}_{\bull,\bull}$ are  upper-unitriangular matrices.  The blocks $b_{-,-}$, $b_{-,-'}$, $b_{+',+}$, and $b_{+,+}$ are lower-triangular, with nonzero diagonal entries.  The blocks $b_{-',-'}$ and $b_{+',+'}$ are diagonal with nonzero entries.  The remaining four (nonzero) blocks are arbitrary $m\times m$ matrices.  
(Again, these are depicted in Figures~\ref{f.borels} and \ref{f.borels2}.)

The image of multiplication $A\times\B'_\circ \times A^- \to Q$ consists of matrices of the form
\[
 \left(\begin{array}{c|c|c|c} a_{-,-} b_{-,-} \tilde{a}_{-,-} & a_{-,-} b_{-,-'} & a_{-,-} b_{-,+'} & a_{-,-} b_{-,+} \tilde{a}_{+,+} \\ \hline  & a_{-',-'} b_{-',-'} & a_{-',-'} b_{-',+'} & a_{-',-'} b_{-',+} \tilde{a}_{+,+} \\ \hline  &  & a_{+',+'} b_{+',+'} & a_{+',+'} b_{+',+} \tilde{a}_{+,+} \\ \hline  &  &  & a_{+,+'} b_{+,+} \tilde{a}_{+,+}\end{array}\right).
\]

We write an element $q\in Q$ by breaking it into blocks $q_{\bull,\bull}$ according to the same pattern.  If $q$ is generic, then the block $q_{-,-'}$ can be written uniquely as $q_{-,-'} = a_{-,-} b_{-,-'}$, a product of an upper-unitriangular matrix and a lower-triangular matrix.  Next, we can uniquely write $a_{-,-}^{-1} q_{-,-} = b_{-,-} \tilde{a}_{-,-}$ as a product of a lower-triangular matrix and an upper-uni{\-}triangular matrix.  Continuing in this way, we uniquely determine all entries of $a$, $\tilde{a}$, and $b$.
\end{proof} 

\section{Examples}\label{s.examples}

In \cite{infschub}, it was shown that the coefficients $\hat{c}_{\mu,v}^w$ can be computed by expanding the product of a Schubert polynomial by another Schubert polynomial in which the $y$-variables are permuted \cite[Corollary~8.10]{infschub}.  The same can be done in K-theory.  Following \cite[\S8]{infschub}, let $w_\mu$ be the Grassmannian permutation associated to the partition $\mu$, and suppose both $w_\mu$ and $v$ lie in $\Sgp_{(-m,m]}$.  A permutation $\mu\oslash_m v$ in $\Sgp_{(-2m,2m]}$ is constructed by setting
\begin{align*}
 \mu\oslash_m v &= [w_\mu(-m+1)-m,\ldots,\,w_\mu(0)-m,\,v(-m+1)+m,\ldots \\ 
 & \qquad \ldots,\, v(m)+m,\,w_\mu(1)-m,\ldots,\,w_\mu(m)-m ].
\end{align*}
Let $x^{(m)} = [-2m+1,\ldots,\,-m,\,1,\ldots,\,2m,\,-m+1,\ldots,\,0]$.

\begin{proposition}\label{p.dirsum-push}
We have
\[
  \boxplus_*[\xi_\mu \times \xi_v] = [\O_{X_{\mu\oslash_m v} \cap \Omega_{x^{(m)}}}(-\partial X_{\mu\oslash_m v} )] = \xi_{\mu\oslash_m v} \cdot [\O_{\Omega_{x^{(m)}}}].
\]
\end{proposition}

This is proved as in \cite[Lemma~2.2]{thomas-yong}; see also \cite[Theorem~7.5]{LRS}.

Specialize the $z$ variables by setting
\begin{align*}
 \mathbb{z} &= (z_{-m+1},\ldots,z_m,z_{-m+1},\ldots,z_m)
 \intertext{and} 
\tilde{\mathbb{z}} &= (z_{-m+1},\ldots,z_0,z_{-m+1},\ldots,z_m,z_{1},\ldots,z_m).
\end{align*}
Let $\mathbb{c}$ and $\tilde{\mathbb{c}}$ be the specializations of $\prod_{i=-2m+1}^0\frac{1+u_i}{1+\tilde{x}_i}$ to $u=\mathbb{z}$ and $u=\tilde{\mathbb{z}}$, respectively.  Using Proposition~\ref{p.dirsum-push}, the proof of the following is analogous to that of \cite[Corollary~8.10]{infschub}:

\begin{corollary}
The coefficient $\hat{d}_{\mu,v}^w(z)$ is equal to the coefficient of $\enG_{\mu\oslash_m v}(\mathbb{c};x;\mathbb{z})$ in the expansion of $\enG_w(\tilde{\mathbb{c}};x;\tilde{\mathbb{z}})\cdot \enG_{x^{(m)}}(\mathbb{c};x;\mathbb{z})$.
\end{corollary}

The following examples use this method, implemented in Maple, to compute $\hat{d}_{\mu,v}^w(z)$.  To compare with the notation of \cite{LLS3}, we have $k^w_\mu(a) = \hat{d}^w_{\mu,e}(z)$ under the substitutions $\beta=-1$ and $z_i = \ominus a_i$.  (In particular, $-z_j\ominus z_i$ goes to $-a_i\ominus a_j$.)

\begin{example}
We have
\[
  \hat{d}^{[2,1,0,-1]}_{(2,2),[1,0,-1]} = (\beta)^2(1+\beta z_{0}\ominus z_2)(z_{-1}\ominus z_1).
\]
Evaluating at $\beta=-1$ and multiplying by $(-1)^{|\mu|+\ell(v)-\ell(w)}=(-1)^{4+3-6}=-1$, we have
\[
  -\hat{d}^{[2,1,0,-1]}_{(2,2),[1,0,-1]}  = ( 1 - z_0\ominus z_2 )(-z_{-1}\ominus z_1) = (\ee^{y_1-y_{-1}}-1) + (\ee^{y_2-y_0}-1)(\ee^{y_1-y_{-1}}-1),
\]
which exhibits the expected positivity.
\end{example}

\begin{example}
We have
\begin{align*}
  \hat{d}^{[2,1,0,-1]}_{(2,1),[1,2,0,-1]} &= \beta\Big[(1+\beta z_0\ominus z_1)(1+\beta z_{-1}\ominus z_2) + (1+\beta z_0\ominus z_1)(1+\beta z_{-1}\ominus z_1) \\
  &\qquad + (1+\beta z_0\ominus z_1)(1+\beta z_0\ominus z_2) + \beta(z_0\ominus z_1)(1+\beta z_0\ominus z_1) \Big].
\end{align*}
\end{example}

\begin{example}
We have
\begin{align*}
 \hat{d}^{[0,-1,2,1]}_{(2,1),e} &= \beta + \beta^2 z_{0}\ominus z_1 \\
 \hat{d}^{[0,-1,2,1]}_{(2),e} &= 1 + \beta z_{0}\ominus z_1 \\ 
 \hat{d}^{[0,-1,2,1]}_{(1,1),e} &= 1 + \beta z_{0}\ominus z_1 \\ 
 \hat{d}^{[0,-1,2,1]}_{(1),e} &=  z_{0}\ominus z_1,
\end{align*}
which agrees with \cite[Example~8.20]{LLS3} at $\beta=-1$.
\end{example}

\begin{example}
We have
\begin{align*}
 \hat{d}^{[3,1,2]}_{(2),e} &= 1 + \beta z_{2}\ominus z_1  & \hat{d}^{[2,3,1]}_{(1,1),e} &= 1 + \beta z_{0}\ominus z_2 \\
 \hat{d}^{[2,0,1]}_{(2),e} &= 1 &  \hat{d}^{[1,2,0]}_{(1,1),e} &= 1 + \beta z_{0}\ominus z_1\\ 
 \hat{d}^{[1,-1,0]}_{(2),e} &= 1 + \beta z_{0}\ominus z_1 &  \hat{d}^{[0,1,-1]}_{(1,1),e} &= 1  \\ 
 \hat{d}^{[0,-2,-1]}_{(2),e} &= 1 + \beta z_{-1}\ominus z_1 & \hat{d}^{[-1,0,-2]}_{(1,1),e} &= 1 + \beta z_{0}\ominus z_{-1}.
\end{align*}
The first of these agrees with \cite[Example~8.21]{LLS3} at $\beta=-1$.
\end{example}

\begin{remark}
As noted in \cite[\S8]{infschub}, the direct sum morphism $\boxplus\colon \Gr(V) \times \Fl(V) \to \Fl(\bV)$ is equivariant with respect to the full torus $\mathds{T} = T\times T$, so one can consider the $\mathds{T}$-equivariant class of $Y=\boxplus(X_\mu\times X_v)$.  In particular, writing $y$ and $y'$ for the characters coming from the first and second factors of $\mathds{T}$, and likewise for $z,z'$, one has ``double'' coefficients $\hat{c}_{\mu,v}^w(y,y')$ and $\hat{d}_{\mu,v}^w(z,z')$ coming from expansions
\begin{align*}
 \boxplus^*[\bbOmega_w] &= \sum_{\mu,v} \hat{c}_{\mu,v}^w(y,y') \, [\Omega_\mu]\times [\Omega_v] 
\intertext{and}
 \boxplus^*[\O_{\bbOmega_w}] &= \sum_{\mu,v} \hat{d}_{\mu,v}^w(z,z')\, [\O_{\Omega_\mu}]\times [\O_{\Omega_v}].
\end{align*}
in $H_{\mathds{T}}^*\Fl(\bV)$ and $K_{\mathds{T}}\Fl(\bV)$, respectively.

However, the pair $(A,\B')$ is {\em not} an $S$-factorization of $\B$, for the full torus $S=\mathds{T}$: the coordinates whose nonvanishing is required to obtain a smooth multiplication map are invariant for the diagonal subtorus, but not for all of $\mathds{T}$ (cf.~Example~\ref{ex.s-fact}).  And there is no analogous positivity statement for the coefficients $\hat{c}_{\mu,v}^w(y,y')$ and $\hat{d}_{\mu,v}^w(z,z')$ which specializes to the positivity asserted in Theorem~\ref{t.main2}; this is demonstrated in \cite[Example~8.7]{infschub}.  (The examples given there are $\hat{c}_{(2,2),e}^{[2,3,-1,0,1]}(y,y') = (y'_1-y_2)(y'_1-y_1)$ and $\hat{c}_{(1,1),[0,2,-1,1]}^{[2,3,-1,0,1]}(y,y') = (y'_1-y_1)$.)
\end{remark}

\begin{remark}
In K-theory, one can also define coproduct coefficients $\hat{p}_{\mu,v}^w$ from the expansion of dual classes, so
\begin{align}
 \boxplus^*\bm\xi^w &= \sum_{\mu,v} \hat{p}_{\mu,v}^w\, \xi^\mu\times \xi^v
\end{align}
in $K_T(\Gr(V)\times \Fl(V))$, where the dual classes are $\bm\xi^w = [\O_{\bbOmega_w}(-\partial)]$, etc.  The same argument as above, applying the second statement in Corollary~\ref{c.mainK}, shows that these coefficients satisfy the same positivity:
\[
  (-1)^{|\mu|+\ell(v)-\ell(w)}\hat{p}_{\mu,v}^w \in \Z_{\geq 0}[\ee^{y_i-y_j}-1 : i\prec j].
\]
If, following the suggestion made in \cite[\S11]{LLS3}, one were to define back stable Grothendieck polynomials $\bsG_w^\partial$ representing these ideal-sheaf basis elements $\xi^w$, their coproduct coefficients would satisfy the expected equivariant positivity.
\end{remark}



%

\end{document}